\documentclass{amsart}
\usepackage[utf8]{inputenc}

\usepackage{tikz} 
\usetikzlibrary{cd}
\usetikzlibrary{patterns,decorations}
\usepackage{amsmath, amssymb}
\usepackage{stmaryrd}
\usepackage{mathtools}
\usepackage[mathcal]{euscript}

\usepackage{dsfont}
\usepackage{float}
\usepackage{enumitem}

\definecolor{darkblue}{rgb}{0,0,0.6}
\usepackage[ocgcolorlinks,colorlinks=true, citecolor=darkblue, filecolor=darkblue, linkcolor=darkblue, urlcolor=darkblue]{hyperref}
\usepackage[capitalize,noabbrev]{cleveref}

\newcommand{\FF}{\mathbb{F}}
\newcommand{\ZZ}{\mathbb{Z}}

\newcommand{\RR}{\mathbb{R}}

\newcommand{\IZ}{\mathbb{Z}}
\newcommand{\op}{{\mathrm{op}}}

\DeclareMathOperator{\id}{id}
\DeclareMathOperator{\aug}{aug}

\DeclareMathOperator{\Aut}{Aut}

\DeclareMathOperator{\gr}{gr}
\DeclareMathOperator{\Hom}{Hom}
\DeclareMathOperator{\thick}{thick}
\DeclareMathOperator{\Ext}{Ext}

\DeclareMathOperator{\PSL}{PSL}

\DeclareMathOperator{\ind}{ind}

\DeclareMathOperator{\res}{res}

\DeclareMathOperator{\perf}{perf}

\DeclareMathOperator{\End}{End}

\DeclareMathOperator{\modcat}{mod}

\numberwithin{equation}{section}
\newtheorem{theorem}[equation]{Theorem}

\newtheorem{corollary}[equation]{Corollary}
\newtheorem{lemma}[equation]{Lemma}
\newtheorem{proposition}[equation]{Proposition}

\theoremstyle{definition}
\newtheorem{example}[equation]{Example}
\newtheorem{definition}[equation]{Definition}

\newtheorem{remark}[equation]{Remark}
\newtheorem{question}[equation]{Question}

\title[Perfect complexes and an equivariant BGG correspondence]{An equivariant BGG correspondence and perfect complexes for extensions by $\ZZ/2\times \ZZ/2$}

\author{Henrik Rüping and Marc Stephan}
\date{April 16, 2024}


\begin{document}



\begin{abstract}
We provide an equivariant extension of Carlsson's BGG correspondence in characteristic two.
As an application we classify perfect cochain complexes of $(\ZZ/2\times \ZZ/2)\rtimes Q$-representations with four-dimensional total homology for finite groups $Q$ of odd order. 
We deduce that cochain complexes of finite, free $A_4$-CW complexes with four-dimensional total homology are rigid: They are determined by the degrees of the nonzero homology groups.
\end{abstract}

\maketitle

\section{Introduction}

A classical problem in the theory of transformation groups asks which finite groups $G$ can act freely on a finite CW complex $X$ homotopy equivalent to a sphere $S^n$. Smith \cite{smith1944} showed that such a group $G$ can not contain an elementary abelian subgroup of rank $2$, i.e., of the form $\ZZ/p\times \ZZ/p$ for any prime $p$. Swan \cite{swan1959} proved that this condition is also sufficient for the existence of $n\geq 1$ such that $G$ acts freely on some $X\simeq S^n$. The problem which dimensions $n$ can occur is still open and related to number theoretic questions; see the survey \cite{hambleton2015}.

The rank conjecture of Benson and Carlson \cite{bensoncarlson1987} states more generally that a finite group $G$ can act freely on a finite CW complex $X$ homotopy equivalent to a product of $r$ spheres $S^{n_1}\times\ldots \times S^{n_r}$ if and only if $G$ does not contain an elementary abelian subgroup of rank $r+1$. For $r=2$, the condition that $G$ can not contain a subgroup of the form $(\ZZ/p)^3$ was already known by work of Heller \cite{heller1954}. For $r=2$ and finite $p$-groups, Adem and Smith \cite{adem2001periodic} showed that it is also sufficient. Nevertheless, it is open whether the simple groups $\PSL_3(\FF_p)$ of rank $2$ for odd primes $p$ can act freely on some $X\simeq S^m\times S^n$. 

In collaboration with Yal{\c c}{\i }n \cite{ruepingstephanyalcin2022}, we investigated which dimensions $m$ and $n$ can occur for free actions of the alternating group $G=A_4$ on $X\simeq S^m\times S^n$ following a question of Blaszczyk \cite{blaszczyk2013free} and extending Oliver's result  \cite[Theorem~2]{oliver1979} that $A_4$ can not act freely on a product of two equidimensional spheres. We proved that any such action yields a parameter ideal with parameters of degrees $m+1$, $n+1$ in the group cohomology ring $H^*(BA_4;\FF_2)$ that is closed under Steenrod operations, and classified these ideals explicitly. These topological obstructions apply to free actions of any finite simple group of rank $2$, as they all contain $A_4$ as a subgroup. In this paper, we approach the study of free $A_4$-actions from a more algebraic side.

Algebraically, we are interested in perfect cochain complexes over the group ring $k[G]$ for a finite group $G$ and field $k$ in modular characteristic with total homology of dimension $2^r$, as the cellular cochain complex of any finite $G$-CW complex homotopy equivalent to a product of $r$ spheres is an example of such a perfect complex. Again, if $r=1$ or $r=2$, then these perfect complexes can only exist if $G$ does not contain a subgroup of the form $(\ZZ/p)^{r+1}$. If the homology is concentrated in one degree, then we are just studying projective modules of dimension $2^r$. For $r=1$ and nonzero homology in two different degrees $m<n$, any bounded above cochain complex $C^*$ is classified by the $k[G]$-representations $H^m(C^*)$, $H^n(C^*)$ and one $k$-invariant element $k(C^*)\in \Ext_{k[G]}^{n-m+1}(H^n(C^*), H^m(C^*))$; see \cite[7.6~Satz]{dold1960}. If $H^*(C^*)$ is finite-dimensional, then the cochain complex $C^*$ is perfect if and only if it has trivial support; see \cite[Theorem~11.4]{bensoniyengarkrause2011}. For $r\geq 2$, we can have nonzero homology in more than two degrees. In general, such cochain complexes are not classified anymore by their $k$-invariants. For the alternating group $A_4=(\ZZ/2\times\ZZ/2)\rtimes C_3$ and, more generally, for any extension of a finite group $Q$ of odd order by $\ZZ/2\times \ZZ/2$, we classify the perfect cochain complexes with four-dimensional total homology.

\begin{theorem}[see~\cref{cor:classification_four_dim_perfect}]\label{thm:classification_four_dim_perfect_intro} Let $\FF$ be a field of characteristic two and let $Q$ be a group of odd order acting on $P=\ZZ/2\times \ZZ/2$. There is a bijection between quasi-isomorphism classes of perfect cochain complexes over $\FF[P\rtimes Q]$ with four-dimensional homology and triples $(l,L,J)$, where $l$ is an integer, $L$ a one-dimensional $Q$-representation, and $J\subset H^*(BP;\FF)$ is a $Q$-invariant parameter ideal.
\end{theorem}

The bijection assigns to a perfect cochain complex $C^*$ with $\dim_\FF H^*(C^*)=4$, the degree $l$ of the lowest nonzero homology group, the $Q$-representation $\Ext^l_{\FF[P]}(\FF, C^*)$, and the annihilator ideal $J\subset \Ext_{\FF[P]}^*(\FF,\FF)$ of $\Ext^*_{\FF[P]}(\FF, C^*)$. Up to shifting, we may assume that $l=0$ and that the four basis elements of $H^*(C^*)$ have degrees $0\leq m\leq n \leq t$. We will show that $t=m+n$ and that the corresponding parameter ideal has parameters of degrees $m+1$ and $n+1$; see \cref{lem:parameteridealFromPerfectComplex}. In fact, $H^*(C^*)$ is isomorphic to $L$ tensored with the exterior algebra $\Lambda(\Sigma^{-1} J/(H^{>0}(BP)J))$ as graded $\FF[Q]$-modules.

The two main methods to establish \cref{thm:classification_four_dim_perfect_intro} hold in much greater generality. In \cref{sec:spectralsequence_extensions}, we extend the spectral sequence of \cite{ruepingstephan2022} from finite $p$-groups $P$ to finite extensions by $P$. We show in \cref{sec:spectralsequence_elementaryab} that the spectral sequence collapses for $C^*$ as in \cref{thm:classification_four_dim_perfect_intro} and establish that the annihilator ideal of $\Ext_{\FF[P]}^*(\FF, C^*)$ is a $Q$-equivariant parameter ideal.

The idea of constructing perfect complexes from parameter ideals in group cohomology is due to Benson and Carlson \cite[Theorem~4.1]{bensoncarlson1994}. Their method constructs perfect cochain complexes with trivial action on homology. In our situation, we also need perfect cochain complexes with nontrivial action on homology. 

In \cref{thm:equivariantBGG} we extend Carlsson's BGG correspondence for perfect cochain complexes over an exterior algebra $\Lambda$ over $\FF$ from \cite{carlsson1986} $Q$-equivariantly. Our formulation holds for any finite group $Q$, not just groups of odd order. 

We use the equivariant BGG correspondence to construct perfect cochain complexes and to establish the classification of \cref{thm:classification_four_dim_perfect_intro} in \cref{sec:classification_perfcomplexes}. This uses an identification of the $Q$-algebra $\FF[P]$ for a group of odd order $Q$ acting on $P=(\ZZ/2)^n$ with an exterior algebra $\Lambda(V)$ for a $Q$-representation $V$.

\cref{sec:crossedproductalgebras} on augmented crossed product algebras $A\ast_\gamma Q$ over a field $k$ provides a uniform treatment for augmented skew group algebras and algebras of group extensions to establish a $Q$-action on $\Ext^*_A(k,C^*)$ for cochain complexes $C^*$ over $A\ast_\gamma Q$.  

The cellular cochain complex of a finite, free $G$-CW complex is not just a perfect complex, but a finite, free cochain complex. In \cref{sec:finiteness_obstruction}, we use Wall's finiteness obstruction to determine which perfect complexes from \cref{thm:classification_four_dim_perfect_intro} are homotopy equivalent to a finite, free one. 

\begin{theorem}[see~\cref{thm:finitefreecomplexes}]\label{thm:finite_free_complex_intro} Let $\FF$ be a field of characteristic two and $C^*$ a perfect cochain complex over $\FF[A_4]$ with four-dimensional homology. Then $C^*$ is homotopy equivalent to a finite, free $\FF[A_4]$-cochain complex if and only if the corresponding parameter ideal $J$ has a $C_3$-invariant parameter of even degree.
\end{theorem}

The modular representation theory of $A_4$ depends on whether the polynomial $X^2+X+1$ is irreducible over $\FF$,  see \cite{dlabringel1989,bennetttennenhauscrawleyboevey2022}, and so does the proof of \cref{thm:finite_free_complex_intro} even though its statement does not. For instance, if $X^2+X+1$ is irreducible, then the only one-dimensional representation is the trivial representation.

For $\FF=\FF_2$ we can read off from the homology of $C^*$ whether it is homotopy equivalent to a finite, free one: A perfect $\FF_2[A_4]$-cochain complex $C^*$ such that its homology is four-dimensional with basis elements in degrees $0\leq m\leq n\leq t$ is homotopy equivalent to a finite, free cochain complex if and only if $m$ or $n$ is odd and $H^*(C^*)$ is a trivial $C_3$-representation, see \cref{cor:finitefreecomplexesforF2}.

\cref{thm:classification_four_dim_perfect_intro} for $\FF=\FF_2$ allows us to count the perfect cochain complexes with four-dimensional homology in fixed degrees by computing $Q$-invariant parameter ideals with parameters in corresponding degrees. For instance, we will compute in future work that up to quasi-isomorphism, there exist $9831$ perfect cochain complexes over $\FF_2[A_4]$ with four-dimensional homology in degrees $0\leq 21\leq 35\leq 56$. Each of these cochain complexes is homotopy equivalent to a finite, free one. A priori, it might seem hard to find a finite, free cochain complex that can not be realized topologically as the cochains on a finite, free $A_4$-CW complex. However, it turns out that it is actually harder to find the ones that can be realized topologically. None of the $9831$ perfect cochain complexes above can be realized topologically, since there does not exist a $C_3$-invariant parameter ideal with parameters of degrees $22$ and $36$ that is closed under Steenrod operations by \cite[Corollary~6.13]{ruepingstephanyalcin2022}. 

Moreover, we know by \cite[Corollary~6.13]{ruepingstephanyalcin2022} that for any degrees $m+1$ and $n+1$, there is at most one Steenrod closed parameter ideal with parameters of degrees $m+1$, $n+1$ in $H^*(BA_4;\FF_2)=H^*(BP;\FF_2)^{C_3}$ for $P=\ZZ/2\times \ZZ/2$. In combination with the current work, this fact has the unexpected consequence that given only the degrees of the nonzero cohomology groups of a finite, free $A_4$-CW complex with four-dimensional cohomology, we can produce its cellular cochain complex up to homotopy equivalence.

\begin{theorem}[see~\cref{thm:rigidity}] If there exists a finite, free $A_4$-CW complex $X$ with four-dimensional total cohomology $H^*(X;\FF_2)$ with a basis in degrees $0\leq m\leq n\leq t$, then its cellular cochain complex is determined by $m$ and $n$ up to homotopy. 
\end{theorem}

\subsection*{Acknowledgments}
It is our pleasure to thank Dave Benson and  Bill Crawley-Boevey for helpful discussions. The research of Stephan was partially funded by the Deutsche Forschungsgemeinschaft (DFG, German Research Foundation) – Project-ID 491392403 – TRR 358.

\section{Augmented crossed product algebras}\label{sec:crossedproductalgebras}
Crossed product algebras are a common generalization of skew group algebras and group algebras of group extensions. A fitting reference for their representation theory is \cite{reitenriedtmann1985}. 
We extend the notion of crossed product algebras to augmented crossed product algebras  $A\ast_{\Psi,\gamma} Q$ for augmented algebras $A$ over a field $k$ and finite groups $Q$. For any cochain complex $C^*$ of right $A\ast_{\Psi,\gamma} Q$-modules, we establish a $Q$-action on $\Ext^*_A(k,C^*)$ and show that the action of $\Ext^*_A(k,k)$ on $\Ext^*_A(k,C^*)$ is $Q$-equivariant.
Similar results for skew group algebras are in \cite{martinez2001} and generalizations thereof to smash product algebras for Hopf algebra actions are in \cite{hevanoystaeyenzhang2011}.
\begin{definition}\label{def:crossedAxioms}
Let $k$ be a field and $(A,\aug\colon A\to k)$ an augmented $k$-algebra. Let $Q$ be a finite group together with a function $\Psi\colon Q\to \Aut(A)$ that assigns to each $q\in Q$ an automorphism $\Psi(q)\colon A\to A$ of augmented algebras. 
Let $\gamma\colon Q\times Q\to U(A)$ be a map to the units of $A$ such that
\begin{enumerate}
    \item\label{crossedAxiomsi} $\gamma(q,q')\gamma(qq',q'')= \Psi(q)(\gamma(q',q''))\gamma(q,q'q'')$ for any $q,q',q''\in Q$, \item\label{crossedAxiomsii}  $\gamma(e, q) = 1 = \gamma(q, e)$ for any $q \in Q$ and $e$ the neutral element of $Q$,
    \item\label{crossedAxiomsiii} $\gamma(q, q') \Psi(qq')(a) = \Psi(q)(\Psi(q')(a)) \gamma(q, q')$ for any $q,q'\in Q$ and $a\in A$,
    \item \label{crossedAxiomsiv} $\aug(\gamma(q,q'))=1$ for all $q,q'\in Q$.
\end{enumerate}
The crossed product algebra $A\ast_{\Psi,\gamma} Q$ (or short $A\ast_{\gamma} Q$) is the $k$-algebra given as a vector space by the free $A$-module on the basis $(\overline{q})_{q\in Q}$ with multiplication
\[a\overline{q}a'\overline{q'} =a\Psi(q)(a')\gamma(q,q')\overline{qq'}.\]
It is augmented by $a\overline{q}\mapsto \aug(a)$.
\end{definition}
We do not require $\Psi$ to be a group homomorphism.
\begin{example}
    If $\Psi$ is a group homomorphism and $\gamma$ constant with value $1$, then $A\ast_{\Psi,\gamma} Q$ is the ordinary skew group algebra $A\ast Q$ of the $Q$-action on $A$.
\end{example}

\begin{example}\label{ex:sesisCrossedProduct}
    Let $N\to G\to Q$ be a short exact sequence of groups and let $s\colon Q\to G$ be a section of sets that preserves the neutral element, i.e., such that $s(e)=e$. For the group ring $A=k[N]$ and the choice of $s$, define
    $$\gamma(q,q')=s(q)s(q')s(qq')^{-1}\quad \text{and}\quad \Psi(q)(a) = s(q)as(q)^{-1}\,.$$
    It is elementary to show that $\gamma$ and $\Psi$ satisfy the conditions of \cref{def:crossedAxioms}. Moreover, we obtain an isomorphism $k[N]\ast_{\Psi,\gamma}Q \to k[G]$ by extending the identity on $k[N]$ via $\overline{q}\mapsto s(q)$.

    If $N\to G\to Q$ splits and $s$ is a section of groups, then this isomorphism specializes to $k[N]\ast Q\cong k[N\rtimes Q]$ writing the group algebra of a semidirect product as a skew-group algebra.
\end{example}

We recall basic identities and properties.
\begin{lemma}\label{lem:crossedsimpleidentities}
   Let $A\ast_{\Psi,\gamma} Q$ be a crossed product algebra. Then
    \begin{enumerate}
        \item \label{crossedsimpleidentitiesi}$\Psi(e) = \id$, 
        \item \label{crossedsimpleidentitiesii}$\overline{e}$ is the neutral element in $A\ast_\gamma Q$,
        \item \label{crossedsimpleidentitiesiii} and each $\overline{q}$ is a unit in $A\ast_\gamma Q$ 
        with inverse $\gamma(q^{-1},q)^{-1}\overline{q^{-1}}=\overline{q^{-1}}\gamma(q,q^{-1})^{-1}$.
    \end{enumerate}
\end{lemma}
\begin{proof}
We have  $\Psi(e)(a)= \Psi(e)(\Psi(e)(a))$ for any $a\in A$ by \cref{def:crossedAxioms}\eqref{crossedAxiomsiii} and \cref{def:crossedAxioms}\eqref{crossedAxiomsii}. Since $\Psi(e)$ is an isomorphism, it follows that $a = \Psi(e)(a)$, showing \eqref{crossedsimpleidentitiesi}.

We have $(a\overline{q})\overline{e}=a\gamma(q,e)\overline{qe}=a\overline{q}$ and 
$\overline{e}a\overline{q} = \Psi(e)(a)\gamma(e,q)\overline{eq} = a\overline{q}$ showing \eqref{crossedsimpleidentitiesii}.

The left inverse of $\overline{q}$ is $\gamma(q^{-1},q)^{-1}\overline{q^{-1}}$ since
$$\gamma(q^{-1},q)^{-1}\overline{q^{-1}}\overline{q}=\gamma(q^{-1},q)^{-1}\gamma(q^{-1},q)\overline{e}=\overline{e}.$$ The right inverse is $\overline{q^{-1}}\gamma(q,q^{-1})^{-1}$ as
$$\overline{q}\overline{q^{-1}}\gamma(q,q^{-1})^{-1}=\gamma(q,q^{-1})\overline{e}\gamma(q,q^{-1})^{-1}=\overline{e}.$$ Hence the two are equal and the inverse of $\overline{q}$. This shows \eqref{crossedsimpleidentitiesiii}.
\end{proof}

The following two lemmas are well-known for skew group algebras. We consider elements in degree $n$ of a hom complex as homomorphisms that are homogeneous of degree $n$.

\begin{lemma}\label{lem:homcomplexQmodule}
    Let $C^*,D^*,E^*$ be right $A\ast_\gamma Q$-cochain complexes. The hom complex $\Hom_A(C^*,D^*)$ is a cochain complex of right $k[Q]$-modules via
    $$(fq)(x)\coloneqq f(x\overline{q}^{-1})\overline{q}\,.$$
   Moreover, composition $\Hom_A(D^*,E^*) \otimes_k \Hom_A(C^*,D^*)  \to \Hom_A(C^*,E^*)$ is $Q$-linear with respect to the diagonal action on the tensor product.
\end{lemma}
\begin{proof}
The map $fq$ is right $A$-linear since
\begin{align*}
(fq)(xa)=&f(x\overline{q}^{-1}\overline{q}a\overline{q}^{-1})\overline{q}
=f(x\overline{q}^{-1}\Psi(q)(a))\overline{q}\\
=&f(x\overline{q}^{-1})\Psi(q)(a)\overline{q}
=f(x\overline{q}^{-1})\overline{q}a
=(fq)(x)a\,.
\end{align*}
The formula defines a $Q$-action as
\begin{align*}
    (f(qq'))(x) =& f(x\overline{qq'}^{-1})\overline{qq'}
     = f(x(\gamma(q,q')^{-1}\overline{q}\overline{q'})^{-1})\gamma(q,q')^{-1}\overline{q}\overline{q'}\\
     =& f(x\overline{q'}^{-1}\overline{q}^{-1}\gamma(q,q'))\gamma(q,q')^{-1}\overline{q}\overline{q'}
     = f(x\overline{q'}^{-1}\overline{q}^{-1})\overline{q}\overline{q'}
     =((fq)q')(x)\,.
\end{align*}
We verify that the differential $d_{\Hom}$ on $\Hom_A(C^*,D^*)$ is $Q$-linear:
\begin{align*}
    &(d_{\Hom}(f)q)(x)\\
    =&((d_D \circ f)-(-1)^{|f|}f\circ d_C)q)(x)
    =d_D(f(x\overline{q}^{-1}))\overline{q}    
    -(-1)^{|f|}f(d_C(x\overline{q}^{-1}))\overline{q} \\
    = &d_D(f(x\overline{q}^{-1})\overline{q}) -(-1)^{|f|}f(d_C(x)\overline{q}^{-1})\overline{q}
    =d_D((fq)(x))-(-1)^{|f|}(fq)(d_C(x))\\
    =&d_{\Hom}(fq)(x)\,.
\end{align*}
Finally, we get for the composition of two composable morphisms $f,f'$:
$$
(fq\circ f'q)(x) = f(f'(x\overline{q}^{-1})\overline{q}\;\overline{q}^{-1})\overline{q}
=f(f'(x\overline{q}^{-1}))\overline{q}
=((f'\circ f)q)(x)\qedhere
$$
\end{proof}

\begin{lemma}\label{lem:tensoroverAisQmodule}
    Let $C^*$ be a right $A\ast_\gamma Q$ cochain complex and $D^*$ be a left $A\ast_\gamma Q$ cochain complex. Then $C^*\otimes_A D^*$ is a cochain complex of right $k[Q]$-modules with $Q$-action given by
    $(x\otimes y)q = x\overline{q}\otimes \overline{q}^{-1}y$.
\end{lemma}
\begin{proof}
The formula for the action of $Q$ is well-defined as 
\begin{align*}    
    xa\overline{q}\otimes_A \overline{q}^{-1}y=&
    x\overline{q}\Psi(q)^{-1}(a)\otimes_A \overline{q}^{-1}y   =x\overline{q}\otimes_A \Psi(q)^{-1}(a)\overline{q}^{-1}y\\
    =&x\overline{q}\otimes_A \overline{q}^{-1}\; \overline{q}\Psi(q)^{-1}(a)\overline{q}^{-1}y
    =x\overline{q}\otimes_A \overline{q}^{-1}ay\,.
\end{align*}
It is indeed a group action since
\begin{align*}
    ((x\otimes_A y)q)q'
    =&x\overline{q} \cdot \overline{q'}\otimes_A \overline{q'}^{-1} \cdot \overline{q}^{-1}y
    =x\gamma(q,q')\overline{qq'}\otimes_A \overline{qq'}^{-1} \gamma(q,q')^{-1}y\\
    =&x\overline{qq'}\Psi(qq')^{-1}(\gamma(q,q'))\otimes_A \Psi(qq')^{-1}(\gamma(q,q')^{-1})\overline{qq'}^{-1} y\\
    =&x\overline{qq'}\otimes_A \overline{qq'}^{-1} y = (x \otimes_A y)(qq')\,.
\end{align*}
We verify that the differential $d_{\otimes}$ on $C^*\otimes_A D^*$ is $Q$-linear:
\begin{align*}
(d_{\otimes}(x\otimes_A y))q
=&(d_C(x)\otimes_A y+(-1)^{|x|}x\otimes_A d_D(y))q\\
=&d_C(x)\overline{q}\otimes_A \overline{q}^{-1}y+(-1)^{|x|}x\overline{q}\otimes_A \overline{q}^{-1}d_D(y)\\
=&d_C(x\overline{q})\otimes_A \overline{q}^{-1}y+(-1)^{|x|}x\overline{q}\otimes_A d_D(\overline{q}^{-1}y)\\
=&d_\otimes(x\overline{q}\otimes_A \overline{q}^{-1}y)
=d_\otimes((x\otimes_A y)q)\qedhere
\end{align*}
\end{proof}

The augmentation on $A\ast_{\gamma}Q$ equips $k$ with an $A\ast_{\gamma}Q$-module structure.

\begin{proposition}\label{lem:equivariantactionOfExt}
For any cochain complex $C^*$ of right $A\ast_\gamma Q$-modules, the graded vector space
$\Ext_A^*(k,C^*)$ has an induced right $Q$-action.
With this action, the map 
\[\Ext_A^*(k,C^*)\otimes_k \Ext_A^*(k,k)\to \Ext_A^*(k,C^*) \]
is a map of graded right $Q$-modules.
\end{proposition}
\begin{proof}
We take a projective resolution of $k$ as a right $A\ast_\gamma Q$-module. Then the claims follow from \cref{lem:homcomplexQmodule} and taking homology.
\end{proof}

\begin{example}
    For a short exact sequence of groups \[N\to G\stackrel{pr}{\to} Q\] as in \cref{ex:sesisCrossedProduct} and $C^*,D^*$ two cochain complexes of right $k[G]$-complexes the $Q$-action on $\Hom_{k[N]}(C^*,D^*)$ can be calculated by \[(fq)(x)=f(xg^{-1})g\,,\] where $g\in G$ is any element such that $pr(g)=q$. 
    In particular, this is independent of the choice of representative for $q$ in $G$.
\end{example}

\section{A spectral sequence for extensions by \texorpdfstring{$p$}{p}-groups}\label{sec:spectralsequence_extensions}
Let $\FF$ be a field of characteristic $p>0$. We generalize the spectral sequence from \cite{ruepingstephan2022} to finite extensions by $p$-groups. Fix a short exact sequence
\[1\rightarrow P \rightarrow G \stackrel{pr}{\rightarrow} Q\rightarrow 1\]
of finite groups such that $P$ is a finite $p$-group. The spectral sequence in \cite{ruepingstephan2022} is obtained from the coradical filtration of a cochain complex $C^*$ of free $\FF [P]$-modules. Here we will show that if $C^*$ is the restriction of a cochain complex of $\FF [G]$-modules, then the spectral sequence becomes a spectral sequence of $\FF [Q]$-modules.

We view $\FF[P]$ as a subring of $\FF[G]$. Let $C^*$ be a cochain complex of right $\FF[G]$-modules such that their restrictions to $\FF[P]$ are free. 

Since group algebras of finite groups over a field are Frobenius algebras, their classes of injective and projective modules agree; see \cite[(15.9)~Theorem]{lam1999}. Moreover, for a finite $p$-group $P$, the algebra $\FF[P]$ is local. Hence projective modules over $\FF[P]$ are free by Kaplansky's theorem on projective modules.

From the topological side, we are interested in cochain complexes of $G$-spaces whose isotropy groups intersect $P$ trivially.
\begin{example}
Let $C_*(X;\FF)$ be the singular chain complex with coefficients in $\FF$ of a $G$-space $X$ such that the restricted $P$-action is free. Then the singular cochain complex $C^*(X;\FF)$ consists of right $\FF[G]$-modules whose restrictions over $\FF[P]$ are free.
\end{example}

From the algebraic side we can dualize perfect chain complexes over $\FF[G]$ or consider perfect cochain complexes directly. 
\begin{example}
Let $C_*$ be a perfect chain complex over $\FF[G]$. Dualizing yields a bounded cochain complex $C^*$ of finitely generated, injective (and thus projective) right $\FF[G]$-modules. Hence $C^*$ is a perfect cochain complex over $\FF[G]$ and since $\FF[P]$ is local, it is free over $\FF[P]$.
\end{example}

Let $I\subset \FF[P]$ be the augmentation ideal. Since $P$ is a $p$-group, the augmentation ideal is nilpotent. We write $L$ for its nilpotency degree minus $1$, i.e., $L$ is the maximal number for which $I^L\neq 0$. We equip $C^*$ with the increasing filtration of $\FF[P]$-modules 
\[0=F^{-1}C^*\subset \ldots \subset F^{L}C^*=C^*,\]
given by 
\[F^iC^* \coloneqq \{ x\mid x\lambda=0\quad\text{for all } \lambda \in I^{i+1}\}.\]

\begin{lemma}\label{lem:filtrationequivariant}
The filtration on $C^*$ is a filtration of right $\FF[G]$-modules.
\end{lemma}
\begin{proof}
We show that right multiplication with any element in $\FF[G]$ preserves the filtration degree. The group $G$ acts on $\FF[P]$ by conjugation. The unique maximal ideal $I$ is invariant by the conjugation action. Thus all powers of $I$ are invariant as well. 
If $x\in F^iC^*$, then for any $g\in G$ and $\lambda\in I^{i+1}$, we obtain
\[xg\lambda = (xg\lambda g^{-1})g=0g=0\]
by the definition of the filtration and since $g\lambda g^{-1}\in I^{i+1}$. Hence $xg$ lies in $F^iC^*$.
\end{proof}

We examine the associated spectral sequence. Since we do not follow standard grading conventions, we briefly recall the pages.
Let $Z_r^{k,t}$ denote the module of $r$-almost cocycles in homological degree $t$ and in filtration degree $k$, i.e.,
\[Z_r^{k,t}\coloneqq \{x\in F^kC^t\mid dx \in F^{k-r}C^{t+1}\}.\]
With this notation, the pages of the spectral sequence are given by 
\[E_r^{k,t}\coloneqq \frac{Z_r^{k,t}}{Z_{r-1}^{k-1,t}+d(Z_{r-1}^{k+r-1,t-1})}.\]
The differential of the cochain complex induces differentials
\[d_r:E_r^{k,t}\rightarrow E_r^{k-r,t+1}\]
on each page. Finally, we have for the induced filtration on $H^*(C^*)$ that \[E_\infty^{k,t} = F^k(H^t(C^*))/F^{k-1}(H^t(C^*)).\] 

\begin{lemma}\label{lem:QactonSS} The $G$-action on $E_r^{k,t}$ descends to a $Q$-action with which the spectral sequence becomes a spectral sequence of $\FF[Q]$-modules.
\end{lemma}
\begin{proof}
If $x$ represents a class $[x]\in E_r^{k,t}$, then $g\in G$ acts on $[x]$ by $[xg]$. We show that this action is independent of the representative of the coset $Pg$ in $G/P\cong Q$. Indeed, for $p\in P$, we obtain
\[[xpg]=[x(p-1)g] + [xg].
\]
Since $(p-1)\in I$, the element $x(p-1)g$ is in $Z^{k-1,t}_{r-1}$ and thus represents zero in $E_r^{k,t}$. Hence the $G$-action descends to a $Q$-action.

Since the differential of $C^*$ is $\FF[G]$-linear, all differentials in the spectral sequence are $\FF[Q]$-linear. 
\end{proof}

In \cite[Proposition~3.6, Corollary~3.7]{ruepingstephan2022}, we calculated the $E_0$-page and $E_1$-page over $\FF[P]$ involving the associated graded $\gr(\FF[P])$ of $\FF[P]$. We will show that these isomorphisms are $Q$-equivariant.
\begin{lemma}\label{lem:qactsongrFP}
The graded ring 
\[\gr(\FF[P])\coloneqq \bigoplus_s I^s/I^{s+1}\]
has a $Q$-action induced by conjugation.
\end{lemma}
\begin{proof}
Each power $I^s$ is invariant under the conjugation action by $G$ by the proof of \cref{lem:filtrationequivariant}. Thus it suffices to show that the action descends to a $Q$-action on $I^s/I^{s+1}$. Let $g\in G$ and $p\in P$. We show that $g$ and $pg$ act the same way. For $\lambda\in I^s$ we obtain
\[[g^{-1}p^{-1}\lambda pg]= [g^{-1}\lambda g] + [g^{-1}(p^{-1}-1)\lambda g]+[g^{-1}\lambda(p-1) g]+[g^{-1}(p^{-1}-1)\lambda(p-1) g]\]
and the last three summands lie in $I^{s+1}$ since $(p-1),(p^{-1}-1)\in I$.
\end{proof}
\begin{lemma}\label{lem:equivisos}
The natural maps
\[E_0^{L,t}\otimes_\FF I^s/I^{s+1} \rightarrow E_0^{L-s,t}, \qquad  [c]\otimes[\lambda]\mapsto [c\lambda],\]
\[E_1^{L,t}\otimes_\FF I^s/I^{s+1} \rightarrow E_1^{L-s,t}, \qquad  [c]\otimes[\lambda]\mapsto [c\lambda],\]
are $Q$-equivariant isomorphisms.
\end{lemma}
\begin{proof}
These are natural isomorphisms by \cite[Proposition~3.6, Corollary~3.7]{ruepingstephan2022}. Since the $Q$ action is descended from $G$, it suffices to establish $G$-equivariance. For $g\in G$, we have
\[[c]g\otimes [\lambda]g = [cg]\otimes [g^{-1}\lambda g] \mapsto [(c\lambda)g]=[c\lambda ]g. \qedhere\]
\end{proof}

To complete the calculation of the $E_0$- and $E_1$-page, we provide $Q$-equivariant identifications for $E_0^{L,*}$ and $E_1^{L,*}$. By \cite[Proposition~3.6]{ruepingstephan2022}, we have \[E_0^{L,t} = C^t/C^tI \cong C^t\otimes_{\FF [P]}\FF\] with $Q$-action descended from right multiplication by $G$ on $C^t$.

\begin{lemma}\label{lem:QactonEL0}
We have a natural $Q$-equivariant isomorphism $E_0^{L,t} \cong C^t\otimes_{\FF [G]}\FF Q$
and thus
\begin{align*}
   E_0^{L-s,t}&\cong (C^t\otimes_{\FF [G]}\FF [Q])\otimes_\FF I^s/I^{s+1}, \\
   E_1^{L-s,t}&\cong H^t(C^*\otimes_{\FF [G]}\FF [Q])\otimes_\FF I^s/I^{s+1}
\end{align*}
as right $\FF[Q]$-modules.
\end{lemma}
\begin{proof}
We let $Q$ act on $\FF[G]\otimes_{\FF[P]}\FF$ as when considering $\FF[G]$ as a cochain complex concentrated in degree zero. Thus $Pg=gP$ acts on $c\otimes 1$ by $cg\otimes 1$. The isomorphism
$\FF[G/P]\cong \FF [G] \otimes_{\FF [P]} \FF $
is $Q$-equivariant and yields a natural isomorphism
\[C^t\otimes_{\FF [P]}\FF\cong C^t\otimes_{\FF [G]} \FF [G] \otimes_{\FF [P]} \FF \cong C^t\otimes_{\FF [G]} \FF [Q] \]
of right $\FF [Q]$-modules.

The identification of the $E_0$- and $E_1$-page follows now from \cref{lem:equivisos}.
\end{proof}

Let $\varepsilon^{*}(G)$ be a projective resolution of $\FF$ by finitely generated projective right $\FF[G]$-modules. We will use the following result for perfect cochain complexes $C^*$.

\begin{lemma}\label{lem:applyingHomOnSS} Suppose that $C^*$ is bounded below. Then
the natural map 
\[C^*\cong \Hom_{\FF}(\FF, C^*)\to \Hom_{\FF}(\varepsilon^*(G), C^*)\]
induces an isomorphism of spectral sequences for the pages $E_{r}$ with $r\geq 1$.
\end{lemma}
\begin{proof}
Since the restriction of $C^*$ to $\FF[P]$ is free, the cochain complex $C^*$ is a bounded below complex of injective $\FF[P]$-modules. The complex $\Hom_{\FF}(\varepsilon^*(G), C^*)$ is bounded below as well and consists of injective (and thus free) $\FF[P]$-modules. Thus the natural map is a quasi-isomorphism between bounded below complexes of injectives and hence a homotopy equivalence. It follows from \cite[Corollary~3.7]{ruepingstephan2022} that the induced map on spectral sequences is an isomorphism.
\end{proof}

We say that a spectral sequence $E$ is a \emph{right module} over a multiplicative spectral sequence $R$ if each page $E_r$ is a bigraded right module over $R_r$ and the Leibniz rule holds for the action $E_r\otimes R_r\to E_r$. Furthermore, we require that the induced multiplication 
$H^*(E_r)\otimes H^*(R_r)\to H^*(E_r) $
agrees with the multiplication on the $(r+1)$-page.
\begin{lemma}\label{lem:spectralsequence_is_module} The spectral sequence $E_{r>0}^{*,*}(\Hom_\FF(\varepsilon^*(G),\varepsilon^*(G)))$ is multiplicative and the spectral sequence $E_{r>0}^{*,*}(\Hom_\FF(\varepsilon^*(G),D^*))$ is a module over it for any cochain complex $D^*$ of right $\FF[G]$-modules.
\end{lemma}
\begin{proof}
The composition 
\[\Hom_{\FF}(\varepsilon^*(G),\varepsilon^*(G))\otimes_\FF \Hom_{\FF}(\varepsilon^*(G),D^*)\to \Hom_{\FF}(\varepsilon^*(G),D^*)\]
is compatible with the filtrations as in the proof of \cite[Proposition~4.4]{ruepingstephan2022}. For $
D^*=\varepsilon^*(G)$, the composition is the multiplication of a filtered differential graded algebra that is compatible with the filtration and thus induces a multiplicative spectral sequence; see \cite{mccleary2001user}[Theorem 2.14]. For arbitrary $C^*$, we obtain analogously an induced module structure of spectral sequences.
\end{proof}

We will use the following description of the $E_1$-page.
\begin{remark}\label{rem:descritionE1page_as_ext} 
By definition of the $E_0$-page we have
\[E_0^{0,*}(\Hom_\FF(\varepsilon^*(G),D^*)) = \Hom_{\FF[P]}(\varepsilon^*(G),D^*)\]
and thus 
\[ E_1^{0,*}(\Hom_\FF(\varepsilon^*(G),D^*))\cong \Ext_{\FF[P]}^*(\FF, D^*). \] For $D^*=\varepsilon^*(G)$, we obtain the group cohomology ring
\[E_1^{0,*}(\Hom_\FF(\varepsilon^*(G),\varepsilon^*(G)))\cong \Ext_{\FF[P]}^*(\FF, \FF) = H^*(BP).\]
All higher pages $E_{r\ge 1}(\Hom_\FF(\varepsilon^*(G),D^*))$ inherit an $H^*(BP)$-module structure 
\[E_{r\ge 1}^{k,t}(\Hom_\FF(\varepsilon^*(G),D^*))\otimes_\FF H^r(BP)\to E_{r\ge 1}^{k,r+t}(\Hom_\FF(\varepsilon^*(G),D^*))\] using the quotient map \[H^*(BP)= E_{1}^{0,*}(\Hom_\FF(\varepsilon^*(G),\varepsilon^*(G))) \to E_{r}^{0,*}(\Hom_\FF(\varepsilon^*(G),\varepsilon^*(G))),\]
and the differentials are $H^*(BP)$-linear.

We will use this structure for $D^*=C^*$ bounded below such that the restriction to $\FF[P]$ is free as in \cref{lem:applyingHomOnSS}. Then we can replace $\Hom_\FF(\varepsilon^*(G),C^*)$ by $C^*$ in the spectral sequence for $E_{r\geq 1}$. The map
\[\Hom_\FF(\varepsilon^*(G),\varepsilon^*(G))\to \Hom_\FF(\varepsilon^*(G),\FF)\]
induces an isomorphism on the $E_1$-page and thus all higher pages as well by \cite[Corollary~3.7]{ruepingstephan2022}. Hence $ E_{r\geq 1}(\Hom_\FF(\varepsilon^*(G),\varepsilon^*(G)))$ can be replaced by the spectral sequence $ E_{r\geq 1}(\Hom_\FF(\varepsilon^*(G),\FF))$.
\end{remark}
\section{The spectral sequence for extensions by \texorpdfstring{$\ZZ/2\times \ZZ/2$}{Z2 x Z2}}\label{sec:spectralsequence_elementaryab}
In this section, we consider the spectral sequence from \cref{sec:spectralsequence_extensions} for an elementary abelian $2$-group $P\cong\ZZ/2\times \ZZ/2$  of rank $2$ and $\FF$ a field of characteristic $2$. If $f_1,f_2$ are generators for $P$, then $\FF[P]$ is an exterior algebra generated by $\lambda_1\coloneqq f_1-1$ and $\lambda_2\coloneqq f_2-1$. The augmentation ideal $I\subset \FF[P]$ is $( \lambda_1,\lambda_2)$ and
\[I^2/I^3\cong \FF\lambda_1\lambda_2,\quad I/I^2\cong \FF\lambda_1\oplus\FF\lambda_2,\quad \FF[P]/I\cong \FF.
\]
Thus only three columns in the spectral sequence can be nontrivial. 
\begin{lemma}\label{lem:spectralsequence_collapse}
Let $C^*$ be cochain complex over $\FF[G]$ such that its restriction to $\FF[P]$ is a perfect complex. If $H^*(C^*)\neq 0$, then $\dim_\FF H^*(C^*)\geq 4$ with
\begin{equation*}
    \dim_\FF E_\infty^{0,*}\geq 1,\quad \dim_\FF E_\infty^{1,*}\geq 2,\quad \dim_\FF E_\infty^{2,*}\geq 1.
\end{equation*}
If $\dim_\FF H^*(C^*)=4$, then the spectral sequence collapses on the $E_2$-page.
\end{lemma}
\begin{proof}
The assumption $H^*(C^*)\neq 0$ implies $H^*(C^*/C^*I)\neq 0$. Indeed, as $\FF[P]$ is a commutative, noetherian, local ring, this can be proved for instance by considering the minimal free resolution of the perfect cochain complex $C^*$ over $\FF [P]$; see e.g.~\cite[Chapter~2, 2.4~Theorem]{roberts1980}.

As graded vector spaces, we have
\[E_1^{0,*}\cong H^*(C^*/C^*I),\quad E_1^{1,*}\cong H^*(C^*/C^*I)^2, \quad E_1^{2,*}\cong H^*(C^*/C^*I).\]
Since the differentials have bidegree $d_r=(-r,1)$,
the lowest nonzero entry of $E_1^{0,*}$ and the highest nonzero entry of $E_1^{2,*}$ survive to $E_\infty$. Setting $d\coloneqq \dim_\FF H^*(C^*/C^*I)$, the total rank of 
$d_1 \colon E_1^{1,*}\rightarrow E_1^{0,*-1}$
is at most $d-1$ and the total rank of
$d_1\colon E_1^{2,*}\rightarrow E_1^{1,*-1}$
is at most $d-1$ as well.
It follows that $\dim_\FF(E_2^{1,*})\ge 2d-(d-1)-(d-1)=2$. For degree reasons, the classes in the middle column $E_2^{1,*}$ cannot support any further differentials and thus survive to the $E_\infty$-page. Together with the two surviving corners from the $E_1$-page, we deduce $\dim_\FF H^*(C)=\dim_\FF E_\infty \geq 4$.

If $\dim_\FF H^*(C)=4$, then $\dim_\FF E_2^{1,*}=2$. It follows that the total ranks of $d_1$ from $E_1^{1,*}$ to $E_1^{0,*}$ and of $d_1$ from $E_1^{2,*}$ to $E_1^{1,*}$ are both $d-1$. Hence $E_2^{0,*}$ consists only of the surviving bottom left corner, $E_2^{2,*}$ consists only of the surviving top right corner, and $d_r=0$ for $r\geq 2$.
\end{proof}

\begin{remark}\label{rem:whereAreTheSurvivingClasses}
If $\dim_\FF(H^*(C^*))=4$ and the degrees of the elements of a homogeneous basis are $0,m,n,t$ with $0\le m\le n \le t$ counted with multiplicities, then the surviving classes on $E_\infty$ sit in bidegrees $(0,0),(1,m),(1,n),(2,t)$. Moreover, $H^0(C^*/C^*I)$ is one-dimensional since $H^0(C^*/C^*I)\cong E_1^{0,0}=E_\infty^{0,0}$.
\end{remark}

Recall that $H^*(BP)\cong \FF[x_1,x_2]$ is a polynomial ring with generators $x_1, x_2$ of degree one.

\begin{proposition}\label{lem:parameteridealFromPerfectComplex}
Let $C^*$ be a perfect $\FF[G]$-cochain complex such that its total homology is four-dimensional with basis elements in degrees $0\leq m\leq n \leq t$. Let $J$ be the annihilator ideal of the graded $H^*(BP)$-module $\Ext_{\FF P}^*(\FF, C^*)$. Then 
\begin{enumerate}
    \item \label{lem:parameteridealFromPerfectComplexi} there is a $Q$-equivariant isomorphism of graded $H^*(BP)$-modules $$\Ext_{\FF P}^*(\FF, C^*)\cong \Ext_{\FF P}^0(\FF, C^*)\otimes_\FF (H^*(BP)/J);$$
    \item \label{lem:parameteridealFromPerfectComplexii} $J$ is generated by a regular sequence of two parameters;
    \item \label{lem:parameteridealFromPerfectComplex3} $H^*(BP)/J$ is a complete intersection; 
    \item \label{lem:parameteridealFromPerfectComplex4} there is an isomorphism of graded $\FF[Q]$-modules
    $$\Ext_{\FF [P]}^0(\FF, C^*)\otimes_\FF \Lambda(\Sigma^{-1} J/(x_1,x_2)J) \cong \gr H^*(C^*);$$
    \item \label{lem:parameteridealFromPerfectComplexv} $t=m+n$.
\end{enumerate}
\end{proposition}
\begin{proof}

We show that the map 
\[E_1^{0,0}(C^*)\otimes_\FF E_1^{0,*}(\Hom_\FF(\varepsilon^*(G),\FF))\to E_1^{0,*}(C^*)\] 
from \cref{lem:spectralsequence_is_module} using \cref{rem:descritionE1page_as_ext} is surjective by induction on the degree. In degree zero, the map is surjective by construction. 
Suppose that the map is surjective in degree $i$ for some $i\geq 0$ and let $z\in E_1^{0,i+1}(C^*)$. Since the spectral sequence collapses on the $E_2$-page by \cref{lem:spectralsequence_collapse} and $E_2^{0,>0}(C^*)$ is zero, the $d_1$-differential surjects onto $E_1^{0,>0}(C^*)$. Thus there exists a class $z'\in E_1^{1,i}(C^*)$ with $d_1(z')=z$. It follows from the induction hypothesis and \cref{lem:equivisos} that there exists $z''\in E_1^{0,0}(C^*)\otimes E_1^{1,i}(\varepsilon^*(G))$ that is mapped to $z'$. Then $d_1(z'')$ is mapped to $d_1(z')=z$ which concludes the induction step.

Since $E_1^{0,0}(C^*)$ is one-dimensional, it follows that $E_1^{0,0}(C^*)\otimes_\FF (H^*(BP)/J)\cong \Ext_{\FF [P]}^*(\FF, C^*)$, showing \eqref{lem:parameteridealFromPerfectComplexi}. We will show that the map 
\begin{align}\label{eq:actionmap}
E^{0,0}_1(C^*) \otimes_\FF E^{*,*}_1(\Hom_\FF(\varepsilon^*(G),\FF))\to E^{*,*}_1(C^*)
\end{align}
on the whole $E_1$-page is surjective. Let $f\colon \varepsilon^*(G)\to C^*$ be a representative of a generator of $E_1^{0,0}(C^*)\cong \Ext_{\FF[P]}^0(\FF, C^*)$. This is a map of right $\FF[P]$-cochain complexes. Postcomposing with $f$ induces a map on spectral sequences \[E_1^{*,*}(\Hom_\FF(\varepsilon^*(G),\varepsilon^*(G)))\to E_1^{*,*}(\Hom_\FF(\varepsilon^*(G),C^*))\] over $\FF$ that is isomorphic to the map from \eqref{eq:actionmap}. By \cref{lem:QactonEL0}, its surjectivity on $E_1^{0,*}$ shows that it is surjective on the whole $E_1$-page.

After shifting the $i$-th column of the $E_1$-page for $\Hom_\FF(\varepsilon^*(G),\FF)$ up by $i$, we obtain the Koszul complex of $H^*(BP)$ by the description of the differential from \cite[Corollary~6.3]{ruepingstephan2022}.

After the same shift, the $E_1$-page for $C^*$ is the tensor product of the module $E_1^{0,0}(C^*)$ with the Koszul complex of $H^*(BP)/J$.

The first Koszul homology of $H^*(BP)/J=\FF[x_1,x_2]/J$ for the sequence $x_1,x_2$ is $J/(x_1,x_2)J$ (see e.g.~\cite[Proof of Lemma~1.4.15]{gulliksenlevin1969}). Since $E_2^{1,*}(C^*)$ is two-dimensional by \cref{lem:spectralsequence_collapse}, so is $J/(x_1,x_2)J$. Thus $J$ is generated by two parameters by the graded Nakayama lemma. Since $H^*(BP)/J$ is finite-dimensional and the Krull dimension of $H^*(BP)$ is $2$, the ideal $J$ is generated by a system of parameters. Since $H^*(BP)$ is Cohen-Macaulay, the two parameters form a regular sequence, showing \eqref{lem:parameteridealFromPerfectComplexii}. The ring $H^*(BP)/J$ is a complete intersection of embedding dimension $2$; see \cite[Theorem~2.3.3]{brunsherzog1993}. Thus \eqref{lem:parameteridealFromPerfectComplex3} holds. 

The Koszul homology of the complete intersection $H^*(BP)/J$ is the exterior algebra $\Lambda(J/(x_1,x_2)J)$; see \cite[Theorem~2.3.11]{brunsherzog1993}. Since the spectral sequence collapses on the $E_2$-page, we obtain that $\gr(H^*(C^*))\cong \bigoplus E_2^{k,*}(C^*)$ is isomorphic to $E_1^{0,0}(C^*) \otimes \Lambda(\Sigma^{-1} J/(x_1,x_2)J)$ showing \eqref{lem:parameteridealFromPerfectComplex4}. 
The generators of $\Sigma^{-1} J/(x_1,x_2)J$ have degrees $m,n$. Thus their exterior product has degree $t=m+n$, showing \eqref{lem:parameteridealFromPerfectComplexv}.
\end{proof}

\begin{remark}\label{rem:43forSpaces}
   If $C^*$ as in \cref{lem:parameteridealFromPerfectComplex} is given by the cochains of a finite, free $G$-CW complex $X$, then the spectral sequence is multiplicative by \cite[Theorem~6.9]{ruepingstephan2022}. The one-dimensional $Q$-representation $\Ext_{\FF [P]}^0(\FF, C^*) \cong H^0(X/P)$ is trivial since the $Q$-action on the space $X/P$ fixes $1\in H^0(X/P)$. In this case \cref{lem:parameteridealFromPerfectComplex} \eqref{lem:parameteridealFromPerfectComplex4} is an isomorphism of graded rings
   \[\Lambda(\Sigma^{-1} J/(x_1,x_2)J) \cong \gr H^*(C^*).
   \]   
\end{remark}

It is not known for which dimensions $m$, $n$ the group $A_4$ can act freely on a finite, CW complex $X$ homotopy equivalent to $S^m\times S^n$. 
\begin{remark} 
The obstruction result
\cite[Theorem~7.5]{ruepingstephanyalcin2022} for $G=A_4$ has the assumption that $X$ is a finite, free $G$-CW complex with cohomology ring $H^*(X;\FF_2)\cong H^*(S^m\times S^n;\FF_2)$ for some $0<m<n$. As explained on \cite[page~31]{ruepingstephanyalcin2022}, it suffices that the total cohomology of $X$ is four-dimensional and such that the product of the two middle classes is the top class.
This last assumption always holds by \cref{rem:43forSpaces}.
\end{remark}

\section{An equivariant BGG correspondence}\label{sec:equivariantbgg}

We will provide an explicit equivariant BGG correspondence in \cref{thm:equivariantBGG}. There are different versions of the BGG correspondence. A related equivariant BGG correspondence is in \cite[Theorem~9.1.2]{floystad2001}. We are interested in Carlsson's from \cite{carlsson1986}. In particular, we consider the exterior algebra as an ungraded algebra and work in characteristic $2$.  Similarly, there are many results on Koszul duality in the literature (see e.g.~\cite{avramov2013}). We have not found a general statement that directly provides our equivariant BGG correspondence.

We begin with two general results for skew group algebras $A\ast Q$. We will use them for $A$ an exterior algebra. 

\subsection{Augmented skew group algebras}
Let $k$ be a field and $A$ an augmented $k$-algebra. Let $Q$ be a finite group acting on the augmented $k$-algebra $A$. We write $\Psi(q)(a)$ for the left action of $q\in Q$ on $a\in A$. Recall that the skew group algebra $A\ast Q$ is the $k$-algebra given by the free $A$-module $\oplus_{q\in Q} A q$ with basis $Q$ and multiplication $(a q)(b p)= (a \Psi(q)(b)) (qp)$ for $a,b \in A$ and $q,p\in Q$. In contrast to the notation for crossed product algebras, we just write $q\in A\ast Q$ instead of $\overline{q}\in A\ast Q$ for $q\in Q$.

Note that $A$ is a right $A\ast Q$-module via $b\cdot(aq)= \Psi(q^{-1})(ba)$. Since this does not hold for crossed product algebras, we restrict to skew group algebras. The right $A\ast Q$-action on $A$ does not commute with the left $A$-action, i.e., $A$ is not an $A$-$A\ast Q$-bimodule and the following result is not just a formal consequence of bimodule structures. We transform the right $Q$-action from \cref{lem:homcomplexQmodule} to a left $Q$-action.
\begin{lemma}\label{lem:AastQstructureOnHoms}
    For any cochain complex $C^*$ of right $A\ast Q$-modules, the hom complex $\Hom_A(C^*,A)$ is a cochain complex of left $A\ast Q$-modules with left $A$-action coming from the $A$-$A$-bimodule structure on $A$ and $Q$-action $qf\coloneqq fq^{-1}$, i.e., $(qf)(x)= f(xq)q^{-1}$.
\end{lemma}
\begin{proof}
A left $A$-module structure and a $Q$-action yield an $A\ast Q$-module structure if $a(qf)=q(\Psi(q)^{-1}(a)f)$. This follows from the computation
\begin{align*}
    &(a(qf))(x)=a(f(xq)q^{-1})=a(\Psi(q)(f(xq)))\\
    =&\Psi(q)(\Psi(q)^{-1}(a)f(xq))=(q(\Psi(q)^{-1}(a)f))(x).\qedhere  
\end{align*}    
\end{proof}

\begin{lemma}\label{lem:Borelastensorproduct2} Let $C^*$ and $P^*$ be cochain complexes of right $A\ast Q$-modules. If $P^*$ is bounded above and consists of finitely generated $A$-projective modules and $C^*$ is bounded below (i.e. $C^n=0$ for all small enough $n$), then we have an isomorphism
\[\Phi\colon C^*\otimes_{A} \Hom_{A}(P^*, A)\to \Hom_{A}(P^*, C^*), \qquad c\otimes f\mapsto c\cdot f(\_) ,
\]
of right $k[Q]$-modules.
\end{lemma}
\begin{proof}
The left-hand side inherits a $k[Q]$-module structure by \cref{lem:tensoroverAisQmodule} and the right-hand side inherits a $k[Q]$-module structure by \cref{lem:homcomplexQmodule}.

We check that $\Phi$ commutes with the $Q$-action:
\begin{align*}
\Phi((c\otimes_A f)q)(x)=&\Phi( (cq\otimes q^{-1}f))(x)= (cq) (\Psi(q^{-1})(f(xq^{-1})) \\
=& c f(xq^{-1})q = (\Phi(c\otimes f)q)(x)
\end{align*}
Note that $\Phi$ is a chain map, thus it suffices to check that $\Phi$ is an isomorphism of graded modules.

For any right $A\ast Q$-module $M$ and any finitely generated, $A$-projective right $A\ast Q$-module $P$, the map
\[
M\otimes_A \Hom_A(P,A)\to \Hom_A(P, M),\qquad m\otimes f \mapsto  m f(\_)
\]
is an isomorphism. 

For a fixed degree $n$, consider the map
$$\Phi^n\colon \bigoplus_{m}C^{n+m}\otimes \Hom_A(P^{m},A)\to \prod_{m} \Hom_A(P^m,C^{m+n}).$$
The finiteness assumptions ensure that the right-hand side is a direct sum and thus the map is a direct sum of isomorphisms as above and hence an isomorphism.
\end{proof}

\begin{remark}\label{rem:opAndtwistedProducts}
We write $A^{\op}$ for the opposite algebra. The $Q$-action on $A$ induces a left action on $A^{\op}$ and $A^\op\ast Q\cong (A\ast Q)^{\op}$ via $a q\mapsto (\Psi(q^{-1})(a)) q^{-1}$. Thus results for right $A\ast Q$-modules can be translated to results for left $A^{\op}\ast Q$-modules. If $A$ is commutative, then $A^{\op}\ast Q=A\ast Q$.
\end{remark}
For instance, $A$ is a left $A\ast Q$-module via $(aq)\cdot b= a\Psi(q)(b)$.
We will also need that for a left $A\ast Q$-module $M$ and a left $kQ$-vector space $V$, the tensor product $M\otimes_k V$ is a left $A\ast Q$-module via
\[(a q)(m\otimes v) = (a qm)\otimes qv
\]
for $a\in A$, $m\in M$, $v\in V$ and $q\in Q$. 

\subsection{Equivariant BGG correspondence}
We specialize to an exterior algebra and work over a field $\FF$ of characteristic two. Let $V$ be a finite-dimensional vector space over $\FF$ with a left $Q$-action denoted by $qv$ for $q\in Q$ and $v\in V$. The group $Q$ acts on the dual vector space $V^*=\Hom_k(V,k)$ by $(qf)(v)=f(q^{-1}v)$. Choose a basis $y_1,\ldots,y_n$ of $V$ and a dual basis $x_1,\ldots x_n$  of $V^*$. Let $\Lambda$ be the exterior algebra on $V$ and $S=\FF[x_1,\ldots,x_n]$ the symmetric algebra on $V^*$. We consider $\Lambda$ as a graded algebra concentrated in degree $0$ and grade $S$ by $\deg(x_i)=1$ for $1\leq i\leq n$. Since $Q$ acts on $V$ and $V^*$, we obtain induced actions on the algebras $\Lambda$ and $S$. In this subsection, we omit the notation $\Psi$ for these actions.

Carlsson \cite{carlsson1986} established an equivalence of derived categories 
\[
\beta\colon D_{\Lambda\text{-}\perf}(\Lambda) \to D^{hf}_{S\text{-}\perf}(S)
\]
from perfect chain complexes over the exterior algebra $\Lambda$ to finitely generated, free $S$-dg modules $M$ with finite-dimensional total homology.
We provide an equivariant extension in \cref{thm:equivariantBGG} replacing $\Lambda$ by $\Lambda\ast Q$ and $S$ by $S\ast Q$.

\begin{lemma}\label{lem:injectiveresolution} The graded tensor product $\Lambda \otimes_\FF S$ with differential \[d(c\otimes f)=\sum_i c y_i\otimes x_i f \] is a $\Lambda$-injective resolution of $\Hom_\Lambda(\FF, \Lambda)$ as left $\Lambda\ast Q$-module. If $Q$ is of odd order, then $\Lambda\otimes_\FF S$ is a $\Lambda\ast Q$-injective resolution. \end{lemma}
\begin{proof}
    Nonequivariantly, it is well-known that $d$ is a differential on $\Lambda\otimes_\FF S$ with homology concentrated in degree zero; see \cite[(II)~Proposition~2]{carlsson1983}. The differential commutes with the action of $\Lambda$ by definition. It commutes with the $Q$-action as well as we show now. The element $\sum_{i=1}^n y_i \otimes x_i$ in $V\otimes V^*$ is fixed by $Q$ since under the equivariant isomorphism of vector spaces $V\otimes V^*\cong \End(V)$, $v\otimes f\mapsto f(-)v$, the sum corresponds to $\id_V$. The map
\[
    V\otimes V^*\otimes \Lambda \otimes S\to \Lambda \otimes S, \quad 
    y\otimes x\otimes c\otimes f \mapsto cy\otimes xf,
\]
is a homomorphism of $\FF [Q]$-modules. It follows that 
\[\sum_i q(c)q(y_i)\otimes q(x_i)q(f)= \sum_i q(c)y_i\otimes x_iq(f)\] 
since the left-hand side is the image of $q(\sum_{i
=1}^n y_i\otimes x_i)\otimes qc\otimes qf$ and the right-hand side is the image of $(\sum_{i=1}^n y_i\otimes x_i)\otimes qc\otimes qf$. We conclude that the differential is $Q$-equivariant:
\begin{align*}
    q(d(c\otimes f)) &= \sum_{i=1}^n  q(c)q(y_i)\otimes q(x_i)q(f) = \sum_{i=1}^n q(c)y_i \otimes x_iq(f) \\
    &= d(q(c) \otimes q(f)) = d(q(c\otimes f)). 
\end{align*}
Thus $\Lambda\otimes_\FF S$ is a cochain complex of $\Lambda\ast Q$-modules. The modules are finitely generated and free over $\Lambda$. Since $\Lambda$ is self-injective, the modules are injective over $\Lambda$ so that $\Lambda\otimes_\FF S$ is a $\Lambda$-injective resolution. We have \[H^0(\Lambda \otimes_\FF S)\cong\Lambda^n(V)\cong \Hom_\Lambda(\FF, \Lambda)\] as left $\Lambda\ast Q$-modules. 

Suppose the order of $Q$ does not divide the characteristic of $\FF$. Then the skew group algebra $\Lambda\ast Q$ is self-injective as well; see \cite[Theorem~1.1 and 1.3]{reitenriedtmann1985}. Moreover, a $\Lambda\ast Q$-module is projective over $\Lambda\ast Q$ if and only if it is projective over $\Lambda$. Thus $\Lambda\otimes_\FF S$ consists of finitely generated projective modules and since $\Lambda\ast Q$ is self-injective, it is indeed an injective resolution.
\end{proof}

To emphasize the twisted differential, we write $\Lambda\tilde\otimes_\FF S$ for the cochain complex from \cref{lem:injectiveresolution}.

\begin{lemma}\label{lem:betafunctoristensorproduct} 
    The hom complex \[\varepsilon^*=\Hom_{\Lambda^{\op}}(\Lambda\tilde{\otimes}_\FF S, \Lambda)\] of left $\Lambda$-module homomorphisms is a $\Lambda$-projective resolution of $\FF$  as right $\Lambda\ast Q$-module. For this $\varepsilon^*$ and any bounded below cochain complex $C^*$ of right $\Lambda\ast Q$-modules, there is a natural isomorphism
    \[\Hom_\Lambda(\varepsilon^*,C^*)\cong C^*\tilde{\otimes}_\FF S,\]
    where $C^*\tilde{\otimes}_\FF S$ is $C^*\otimes_\FF S$ with differential given by $d(c\otimes f)=(dc)\otimes f+\sum_i c y_i\otimes x_i f$.
\end{lemma}

\begin{proof}
   The cochain complex $\varepsilon^*$ consists of finitely generated $\Lambda$-projective right $\Lambda\ast Q$-modules by \cref{lem:AastQstructureOnHoms} and \cref{rem:opAndtwistedProducts}. It is a projective resolution of \[\Hom_{\Lambda^\op}(\Hom_\Lambda(\FF,\Lambda), \Lambda)\cong \Hom_{\Lambda^\op}(\Lambda^n(V), \Lambda) \cong \FF\] by \cref{lem:injectiveresolution} as $\Lambda$ is self-injective.

    Finally, using \cref{lem:Borelastensorproduct2}, we have a natural isomorphism
\[
    \Hom_{\Lambda}(\varepsilon^*, C^*) \cong C^*\otimes _\Lambda \Hom_\Lambda(\varepsilon^{*}, \Lambda) \cong C^*\otimes_\Lambda \Lambda\tilde{\otimes}_\FF S \cong C^*\tilde{\otimes}_\FF S\]

under which the differential of $C^*\tilde{\otimes}_\FF S$ is as given in the statement.
\end{proof}

Note that $\varepsilon^*=\Hom_{\Lambda^\op}(\Lambda\tilde{\otimes}_\FF S, \Lambda)$ is a right dg module over $(S^\op\otimes_\FF \Lambda)\ast Q$ equipped with the trivial differential.

\begin{remark}\label{rem:isomorphicresolutions} Nonequivariantly, $C^*\tilde{\otimes}_\FF S$ is $\beta(C^*)$ for Carlsson's functor $\beta$ from \cite[Section~(II)]{carlsson1983}. To compare to \cite{avramovbuchweitziyengarmiller10,avramovbuchweitziyengarmiller10c}, there is an isomorphism \[\Hom_{\Lambda^\op}(\Lambda\tilde{\otimes}_\FF S, \Lambda)\cong \Hom_\FF(\Hom_\FF(\Lambda,\FF)\tilde{\otimes}_\FF S, \FF)\cong \Lambda\tilde{\otimes}_\FF \Hom_\FF(S, \FF)\]
    of dg modules over $(S^{\op}\otimes_\FF \Lambda)\ast Q$, where the differential on the right-hand side is $d(\lambda\otimes h)= \sum_i \lambda y_i\otimes h x_i$. 
\end{remark}

The functor $\Hom_\Lambda(\varepsilon^*,-)$ from cochain complexes of right $\Lambda\ast Q$-modules to right dg modules over $S\ast Q$ is right adjoint to $M\mapsto M\otimes_S \varepsilon^*$. Here $Q$ acts on $M\otimes_S \varepsilon^*$ diagonally.

The counit
$\Hom_{\Lambda}(\varepsilon^*, C^*)\otimes_S \varepsilon^*\to C^*
$ is given by evaluation $(f_i)_i\otimes \varphi\mapsto f_m(\varphi)$ for $\varphi\in \varepsilon^m$.

The unit $M\to \Hom_{\Lambda}(\varepsilon^*,M\otimes_S\varepsilon^*)$
sends $m\in M_n$ to $m\otimes_S- \colon \varepsilon^*\to M_n\otimes_S \varepsilon^*$.

\begin{lemma}\label{lem:unicounitqi}
    For $M=S$, the unit $S\to \Hom_{\Lambda}(\varepsilon^*,\varepsilon^*)$ is a quasi-isomorphism. For $C^*=\FF$ concentrated in degree zero, the counit $Hom_\Lambda(\varepsilon^*,\FF)\otimes_S \varepsilon^* \to \FF$ is a quasi-isomorphism.
\end{lemma}
\begin{proof} We show that the unit is a quasi-isomorphism in $M=S$. Using \cref{rem:isomorphicresolutions}, consider the composite
 \[S\to \Hom_{\Lambda}(\varepsilon^*,\varepsilon^*)\simeq \Hom_{\Lambda}(\varepsilon^*,\FF) \cong \Hom_{\Lambda}(\Lambda\tilde \otimes_\FF \Hom_\FF(S, \FF),\FF).\] The target has trivial differential and is isomorphic to $\Hom_\FF(\Hom_\FF(S,\FF),\FF)\cong S$. It follows that the unit is a quasi-isomorphism in $M=S$ since the whole composite $S\to S$ is the identity.

  For $C^*=\FF$, the counit factors as an isomorphism followed by the resolution $\varepsilon^*\to \FF$:
 \[ \Hom_\Lambda(\varepsilon^*,\FF)\otimes_S \varepsilon^*\cong S\otimes_S \varepsilon^* \to \FF\]
 Hence, the counit is a quasi-isomorphism in $\FF$.
\end{proof}

For a differential graded algebra $R$, we write $K(R)$ for the homotopy category of differential graded right modules over $R$ and $D(R)$ for the corresponding derived category obtained by localization with respect to the quasi-isomorphisms. We equip $K(R)$ and $D(R)$ with the usual triangulated structures. For an object $X$ of $D(R)$, we write $\thick(X)$ (or $\thick_R(X)$) for the thick subcategory of $D(R)$ generated by $X$, i.e., the intersection of all, full triangulated subcategories of $D(R)$ that contain $X$ and are closed under taking summands.  We use cohomological grading. In particular, the derived category of cochain complexes over $\Lambda\ast Q$ is $D(\Lambda\ast Q)$. 

The adjoint functors $-\otimes_S \varepsilon^*$ and $\Hom_{\Lambda}(\varepsilon^*,-)$ induce exact, adjoint functors on the triangulated homotopy categories
\[K(S\ast Q)\rightleftarrows K(\Lambda\ast Q).\]
The functor $\Hom_{\Lambda}(\varepsilon^*,-)$ preserves quasi-isomorphisms and thus has a right derived functor $R\Hom_{\Lambda}(\varepsilon^*,-)$ with $R\Hom_{\Lambda}(\varepsilon^*,C^*)=\Hom_{\Lambda}(\varepsilon^*,C^*)$ for $C^*\in D(\Lambda\ast Q)$. The left derived functor of $-\otimes_S \varepsilon^*$ exists as well and can be computed in a dg module $M$ over $S\ast Q$ by applying $-\otimes_S \varepsilon^*$ to a semifree resolution of $M$. If the underlying graded $S$-module of $M$ is finitely generated and free, then $M$ is semifree over $S$, and $M\otimes^L_S \varepsilon^*=M\otimes_S\varepsilon^*$.

The functor $-\otimes_S^L \varepsilon^*\colon D(S\ast Q)\to D(\Lambda\ast Q)$ is exact and has exact right adjoint $R\Hom_{\Lambda}(\varepsilon^*,-)$.

We write $D_{S\text{-}\perf}(S\ast Q)$ for the full triangulated subcategory of $D(S\ast Q)$ of objects isomorphic to dg $S\ast Q$-modules such that the underlying graded $S$-module is free and finitely generated. Moreover, we denote the full triangulated subcategory of $M\in  D_{S\text{-}\perf}(S\ast Q)$ with $\dim_\FF H^*(M)<\infty$ by $D^{hf}_{S\text{-}\perf}(S\ast Q)$. We identify the bounded derived category $D^b(\modcat_{\Lambda\ast Q})$ of finitely generated right $\Lambda\ast Q$-modules with the full subcategory of $D(\Lambda\ast Q)$ consisting of the objects isomorphic to bounded cochain complexes of finitely generated modules. Equivalently, this is the full subcategory of cochain complexes in $D(\Lambda\ast Q)$ with finite-dimensional total homology; see e.g. \cite[Example~4.2.18]{krause22}.

We write $D_{\Lambda\text{-}\perf}(\Lambda\ast Q)$ for the full triangulated subcategory of $D(\Lambda\ast Q)$ of objects isomorphic to bounded cochain complexes whose underlying $\Lambda$-modules are finitely generated and projective. By \cite[Lemma~3.3]{lau2023}, this category agrees with the full subcategory of objects in $D(\Lambda\ast Q)$ that are perfect in $D(\Lambda)$.

\begin{theorem}\label{thm:equivariantBGG}
    The adjunction $(-\otimes_S^L \varepsilon^*, R\Hom_{\Lambda}(\varepsilon^*,-))$ restricts to equivalences of triangulated categories

    \[D_{S\text{-}\perf}(S\ast Q) \rightleftarrows D^b(\modcat_{\Lambda\ast Q})\]
    and
     \[D^{hf}_{S\text{-}\perf}(S\ast Q)\rightleftarrows D_{\Lambda\text{-}\perf}(\Lambda\ast Q).\]
\end{theorem}

\begin{proof}

For any bounded complex $C^*$ of finitely generated right modules over $\Lambda\ast Q$, the dg module $\Hom_{\Lambda}(\varepsilon^*, C^*)$ is $S$-free and finitely generated by \cref{lem:betafunctoristensorproduct}. Thus $R\Hom_{\Lambda}(\varepsilon^*, -)$ restricts to a functor $D^b(\modcat_{\Lambda\ast Q})\to  D_{S\text{-}\perf}(S\ast Q)$. Moreover, the counit of the derived adjunction
$\Hom_{\Lambda}(\varepsilon^*, C^*)\otimes^L_S \varepsilon^*\to C^*$
can be computed via the ordinary counit. We show that it is a quasi-isomorphism. Forgetting the $Q$-action, this counit agrees with the counit for the adjunction with $Q$ the trivial group. This counit is a quasi-isomorphism since it is a quasi-isomorphism for $C^*=\FF$ by \cref{lem:unicounitqi} and since $\thick(\FF)\subset D(\Lambda)$ is the bounded derived category of finitely generated $\Lambda$-modules. Hence, the counit between the derived adjunction is an isomorphism for $C^*\in D^b(\modcat_{\Lambda\ast Q})$.

If a dg module $M\in D(S\ast Q)$ is $S$-free and finitely generated, then the unit
$M\to \Hom_{\Lambda}(\varepsilon^*,M\otimes^L_S\varepsilon^*)$
can be computed by the ordinary unit. Forgetting the $Q$-action, this unit agrees with the unit for the adjunction with $Q$ the trivial group. It is a quasi-isomorphism since it is a quasi-isomorphism for $M=S$ by \cref{lem:unicounitqi} and $\thick(S)\subset D(S)$ contains all dg modules that are free and finitely generated. Hence, the unit between the derived adjunction is an isomorphism for $M\in D_{S\text{-}\perf}(S\ast Q)$.

To establish the first equivalence, we are left to show that $-\otimes_S^L \varepsilon^*$ restricts to a functor $D_{S\text{-}\perf}(S\ast Q)\to D^b(\modcat_{\Lambda\ast Q})$. Forgetting the $Q$-action, this reduces to the case $M=S$ for which $S\otimes_S^L \varepsilon^*\cong \FF$. Hence $M\otimes_S^L \varepsilon^*$ belongs to $D^b(\modcat_{\Lambda\ast Q})$. 

To establish the second equivalence, we check that the adjoint functors restrict further. If $C^*$ is a bounded complex of finitely generated modules over $\Lambda\ast Q$ that are $\Lambda$-projective and thus injective over $\Lambda$, then $H^*(\Hom_{\Lambda}(\varepsilon^*, C^*))\cong H^*(\Hom_\Lambda(\FF,C^*))$ has finite total dimension. 

On the other hand, if $M$ over $S\ast Q$ is $S$-free and finitely generated with finite-dimensional total homology, then we will show that $M\otimes_S\varepsilon^*$ is perfect over $\Lambda$. This holds since for trivial $Q$ the equivalence restricts to an equivalence $\thick_\Lambda(\Lambda)\simeq \thick_{S}(R\Hom_{\Lambda}(\varepsilon^*,\Lambda))=\thick_{S}(\FF)$,
and $\thick_{S}(\FF)=D^{hf}_{S\text{-}\perf}(S)$ by \cite[Theorem~6.4]{avramovbuchweitziyengarmiller10}. 
We have shown that the first equivalence restricts as claimed, providing the second equivalence.
\end{proof}

If $Q$ is of odd order, then a right $\Lambda\ast Q$-module is projective if and only if it is projective over $\Lambda$. Thus we obtain the following consequence.
\begin{corollary}\label{cor:equivalenceforQodd}
    If $Q$ is of odd order, then the adjunction $(-\otimes_S^L \varepsilon^*, R\Hom_{\Lambda}(\varepsilon^*,-))$ restricts to equivalences of triangulated categories
     \[D^{hf}_{S\text{-}\perf}(S\ast Q)\rightleftarrows D_{\perf}(\Lambda\ast Q),\]
     where the latter category is the perfect derived category of right $\Lambda\ast Q$-modules.
\end{corollary}

\begin{remark} 
Carlsson used homological grading in \cite{carlsson1986}. Apart from that, our functor $R\Hom_{\Lambda}(\varepsilon^*,-)\colon D_{\Lambda-\perf}(\Lambda)\to D^{hf}_{S\text{-}\perf}(S)$ for $Q$ the trivial group agrees with the functor $H$ from \cite[Theorem~II.7]{carlsson1986}. This follows from \cref{lem:betafunctoristensorproduct}. For a free, finitely generated dg $S$-module $M$ with $\dim_\FF H^*(M)<\infty$, we have $M\otimes^L_S \varepsilon^* \cong M\otimes_S (\Lambda\tilde{\otimes}_\FF \Hom_\FF(S,\FF))$ instead of $G(M)=M\tilde{\otimes}_\FF \Lambda$ from the proof of \cite[Theorem~II.7]{carlsson1986}. For a perfect $\Lambda$-cochain complex $C^*$, the composite
\[G(H(C^*))= (C^*\tilde{\otimes}_\FF S)\tilde{\otimes}_\FF \Lambda\to C^*\otimes_\FF \Lambda \to C^*
\] induced by the augmentation $S\to \FF$ and the structure map for $C^*$ over $\Lambda$ is not a chain homotopy equivalence. Indeed, applying the functor $-\otimes_\Lambda \FF$ to this map and precomposing with the quasi-isomorphism $\Hom_\Lambda(\FF,C^*)\simeq C^*\tilde{\otimes}_\FF S$ yields the zero map $\Hom_\Lambda(\FF,C^*)\to C^*\tilde{\otimes}_\Lambda \FF$.
\end{remark}

\begin{remark}\label{rem:identHBPandA} For $C^*=\FF$ concentrated in degree
zero, the isomorphism \[
\Ext_{\Lambda}^*(\FF,\FF)\cong H^*(\Hom_\Lambda(\varepsilon^*,\FF))\cong S
\]
is a graded ring isomorphism, which is compatible with the $Q$-actions. Moreover, for an arbitrary cochain complex $C^*$ over $\Lambda\ast Q$, the $\Ext^*_{\Lambda}(\FF,\FF)$-action on $\Ext^*_{\Lambda}(\FF,C^*)$ from \cref{lem:equivariantactionOfExt} agrees with the $S$-action on $H^*(\Hom_\Lambda(\varepsilon^*, C^*))$; see \cite[Section~1.3]{alldaypuppe1993}.
\end{remark}

\section{Perfect complexes with small homology for extensions by 
\texorpdfstring{$(\ZZ/2)^2$}{(Z/2)²}}

\label{sec:classification_perfcomplexes}

In this section we restrict the equivalence $R\Hom_\Lambda(\varepsilon^*,-)$ from \cref{cor:equivalenceforQodd} to perfect cochain complexes over $\Lambda \ast Q$ with four-dimensional total homology for $\Lambda=\Lambda(y_1,y_2)$ and $S=\FF[x_1,x_2]$. In particular, we assume that $Q$ is of odd order. We will provide an explicit classification of these perfect complexes with small homology in \cref{thm:classification_four_dim_perfect}.

\begin{lemma}\label{lem:bijectionfourdim}
The functor $\Hom_{\Lambda}(\varepsilon^*,-)$ induces a bijection between isomorphism classes of perfect dg modules $C$ over $\Lambda\ast Q$ with four-dimensional homology in degrees $0\leq m\leq n \leq l$ and isomorphism classes of objects $M$ of $D^{hf}_{S\text{-}\perf}(S\ast Q)$ such that $M\otimes_S \FF$ has four-dimensional homology in the same degrees. Moreover, if $J\subset S$ is the annihilator ideal of $H^*(\Hom_{\Lambda}(\varepsilon^*,C))$, then $J$ is a $Q$-invariant parameter ideal, $H^*(\Hom_{\Lambda}(\varepsilon^*,C))\cong H^0(\Hom_{\Lambda}(\varepsilon^*,C))\otimes_\FF S/J$ as right $S\ast Q$-modules, and $H^0(\Hom_{\Lambda}(\varepsilon^*,C))$ is one-dimensional over $\FF$.
\end{lemma}
\begin{proof}
The bijection holds by \cref{thm:equivariantBGG} and since $\Hom_{\Lambda}(\varepsilon^*, C)\otimes_S \FF\cong  C$ as cochain complexes.

The ideal $J$ is a parameter ideal by 
\cref{lem:parameteridealFromPerfectComplex} considering $\Lambda$ as $\FF[P]$, and it is $Q$-invariant by \cref{lem:equivariantactionOfExt}. By \cref{lem:parameteridealFromPerfectComplex}\eqref{lem:parameteridealFromPerfectComplexi}, the $S$-action \[H^*(\Hom_{\Lambda}(\varepsilon^*,C))\otimes_\FF S\to H^*(\Hom_{\Lambda}(\varepsilon^*,C))\] induces an isomorphism \[H^0(\Hom_{\Lambda}(\varepsilon^*,C))\otimes_\FF S/J\cong H^*(\Hom_\Lambda(\varepsilon^*,C))\] of graded $S$-modules. Since the $S$-action is $Q$-equivariant, so is the induced isomorphism. Finally, $H^0(\Hom_{\Lambda}(\varepsilon^*,C))$ is one-dimensional By \cref{rem:whereAreTheSurvivingClasses}, .
\end{proof}

There is a canonical ring homomorphism $\FF [Q] \to S\ast Q$ as the inclusion of the degree zero part. Then for any graded right $\FF [Q]$-module $W$ the induced module $\ind_{\FF [Q]}^{S\ast Q}(W)$ is just $W\otimes_\FF S$ with the diagonal right $Q$-action. For a graded right $S\ast Q$-module $M$, we will use repeatedly the induction-restriction adjunction
\[\Hom_{S\ast Q}(W\otimes_\FF S, M)\cong \Hom_{\FF [Q]}(W, M)
\]
to extend $\FF [Q]$-linear maps $W\to M$ to $S\ast Q$-linear maps $W\otimes_\FF S\to M$.

\begin{lemma}\label{lem:bijectionfourdimforM}
There is a bijection between isomorphism classes of objects $M$ in $D^{hf}_{S\text{-}\perf}(S\ast Q)$ such that $M\otimes_S \FF$ has four-dimensional homology in degrees $0\le m\le n\leq t$, and pairs $(L,J)$ of a one-dimensional $Q$-representation $L$ and a $Q$-invariant parameter ideal $J$ in $S$ with parameters in degrees $m+1,n+1$.
The map assigns to an object $M$ in $D^{hf}_{S\text{-}\perf}(S\ast Q)$ its zeroth homology $H^0(M)$ and the annihilator ideal of $H^*(M)$. In particular, there is no such $M$ with four-dimensional homology unless $l=m+n$.
\end{lemma}

\begin{proof}
If $M$ is an object of $D^{hf}_{S\text{-}\perf}(S\ast Q)$ such that $H^*(M\otimes_S \FF)$ is four-dimensional, then by \cref{lem:bijectionfourdim}, the annihilator ideal $J$ in $S$ of $H^*(M)$ is a $Q$-invariant parameter ideal and $H^*(M)\cong S/J\otimes H^0(M)$ with $H^0(M)$ one-dimensional. 

Now suppose that $L$ is a one-dimensional $Q$ representation over $\FF$ and that $J$ is such a parameter ideal. Then $J/(x_1,x_2)J$ is a two-dimensional $\FF [Q]$-module. The projection $J\to J/(x_1,x_2)J$ is an $\FF [Q]$-linear map. Since $Q$ is of odd order, it has a section $\sigma\colon J/(x_1,x_2)J \to J$.

We extend $\sigma$ to the map $J/(x_1,x_2)J\otimes S\to J\otimes S \to S$ and let $K$ be the Koszul complex of this map tensored with $L$, i.e.,
$$ K = L\otimes \Lambda (J/(x_1,x_2)J) \otimes S $$ 
with differential $l\otimes j_1 \wedge \ldots \wedge j_r \otimes s \mapsto \sum_{0\le k \le r} l\otimes j_1 \wedge \ldots\wedge \widehat{j_k}\wedge\ldots \wedge j_r \otimes \sigma(j_k)s$.
Since $\sigma$ is $\FF [Q]$-linear, the differential in $K$ is $S\ast Q$-linear.
We consider $K$ as an $S\ast Q$-dg module graded by $\deg(l\otimes j_1 \wedge \ldots \wedge j_r \otimes s)=(\sum_k \deg(j_k)) - r +\deg(s)$.

Note that $K\otimes_S \FF=L\otimes \Lambda(J/(x_1,x_2)J)$ with zero differential and grading shift as above. Since $J/(x_1,x_2)J$ is a two-dimensional graded $\FF$-module with generators in degrees $m+1,n+1$, the tensor product $L\otimes \Lambda(J/(x_1,x_2)J) $ is four-dimensional with generators in degrees $0,m,n,m+n$.

If $r_1, r_2$ is a basis of $J/(x_1,x_2)J$, then $\sigma(r_1), \sigma(r_2)$ is a minimal generating set of the parameter ideal $J$ and thus a regular sequence. It follows that 
$H^*(K)\cong L\otimes S/J$ by \cite[Corollary~1.6.14]{brunsherzog1993}.

Thus for an arbitrary pair $(L,J)$, we have constructed an object $K \in D^{hf}_{S\text{-}\perf}(S\ast Q)$ such that $H^*(K\otimes_S \FF)$ is four-dimensional, $H^0(K)\cong L$, and such that $J$ is the annihilator ideal of $H^*(K)$. 

Now consider $M\in D^{hf}_{S\text{-}\perf}(S\ast Q)$ such that $H^*(M\otimes_S \FF)$ is four-dimensional.
Let $L=H^0(M)$ and let $J$ be the annihilator ideal of $H^*(M)$. We construct a quasi-isomorphism $K\to M$, where $K$ is constructed as above. Let $W$ denote the graded $Q$-representation $J/(x_1,x_2)J$.

The $S$-module structure induces an isomorphism $ H^0(M)\otimes S/J \cong H^*(M)$ by \cref{lem:bijectionfourdim}. 
Denoting the differential of $M$ by $d$, consider a section $H^*(M)\to \ker d$ of graded $\FF [Q]$-modules. We obtain an $\FF[Q]$-linear map $f_0\colon L\otimes \Lambda^0(W) =H^0(M)\otimes \FF \to M$. By adjunction we can extend it to an $S\ast Q$-linear map $f_0\colon L\otimes \Lambda^0(W)\otimes S\to M$.

Next we construct a map $L\otimes \Lambda^1(W)\otimes S =L\otimes W\otimes S \to M$. Consider the map $L\otimes J/(x_1,x_2)J\to M$ given by $l\otimes j\mapsto f_0(l)\cdot \sigma(j)$. Since $H^*(M)\cong L\otimes S/J$ by \cref{lem:bijectionfourdim}, it follows that this map hits only boundary elements in $M$. Thus it can be lifted to an $\FF [Q]$-linear map $f_1\colon L\otimes W\to M$ satisfying $d\circ f_1(l\otimes w)=f_0(l)\cdot \sigma(w)$. Again by adjunction, we extend it to an $S\ast Q$-linear map $f_1\colon L\otimes \Lambda^1(W)\otimes S\to M$. 

Consider the $\FF[Q]$-linear map $$g\colon L\otimes \Lambda^2(W)\to M, l\otimes (w_1\wedge w_2)\mapsto f_1(l\otimes w_1)\cdot \sigma(w_2) + f_1(l\otimes w_2)\cdot\sigma(w_1).$$ 
The boundary of the right-hand side is $$df_1(l\otimes w_1)\cdot \sigma(w_2) + df_1(l\otimes w_2)\cdot \sigma(w_1)
=f_0(l)\sigma(w_1)\sigma(w_2) + f_0(l)\sigma(w_2)\sigma(w_1)=0.$$ 
Since $\Lambda^2(W)$ is one-dimensional with a generator in degree $m+n+2$ and $f_1$ lowers degrees by one, the image of $g$ is concentrated in degree $m+n+1$. By \cite[Theorem~5.4.1]{neuselsmith2002},  $S/J$ is a Poincar\'e duality algebra with fundamental class in degree $m+n$, in particular $H^{m+n+1}(M)=0$.
Thus the image of $g$ is contained in the boundaries of $M$ and since $\FF [Q]$ is semisimple we can lift $g$ to a map $f_2\colon L \otimes \Lambda^2(W)\to M^{m+n}$ with $df_2=g$ and extend it to an $S\ast Q$-linear map $f_2\colon L\otimes \Lambda^2(W)\otimes S\to M$.

Now define 
$$F\colon L\otimes \Lambda(W) \otimes S=L\otimes (\Lambda^2(W)\oplus \Lambda^1(W)\oplus \Lambda^0(W))\otimes S\to M$$
by $F=(f_2,f_1,f_0)$.
A straightforward computation shows that $F$ is a chain map. By construction, $F$ induces an isomorphism on $H^0$. It follows that $H^*(F)$ is an isomorphism since it is the composite
\[H^*(K)\cong S/J\otimes H^0(K)\cong S/J\otimes H^0(M)\cong H^*(M).\]

Note that $H^*(K\otimes_S \FF)=L\otimes \Lambda(W)$, and thus $t=m+n$. 
\end{proof}
We deduce the main result of this section.
\begin{theorem}\label{thm:classification_four_dim_perfect} There is a bijection between isomorphism classes of perfect dg modules over $\Lambda \ast Q$ with four-dimensional homology and triples $(l,L,J)$ where $l$ is an integer, $J\subset \FF[x_1,x_2]$ is a $Q$-invariant parameter ideal and $L$ is a one-dimensional $Q$-representation. 
\end{theorem}
\begin{proof}
    For a perfect dg module $C^*$ over $\Lambda \ast Q$ with four-dimensional homology, let $l\in \ZZ$ be the lowest degree in which $H^*(C)$ is nonzero. 
    
    Let $L=H^l(\Hom_{\Lambda}(\varepsilon^*,C))=\Ext^l_\Lambda(\FF,C)$.
    Let $J\subset \FF[x_1,x_2]\cong \Ext^*_\Lambda(\FF,\FF)$ be the annihilator ideal of $\Ext^*_\Lambda(\FF, C)\cong H^*(\Hom_{\Lambda}(\varepsilon^*,C))$.
    We show that the assignment $C\mapsto (l,L,J)$ is a bijection as desired.
    
    After shifting $C$, we may assume that $l=0$. Then combining \cref{lem:bijectionfourdim} with \cref{lem:bijectionfourdimforM} shows that the assignment is indeed a bijection as claimed.
\end{proof}

For $Q$ acting on $P=\ZZ/2\times \ZZ/2$, we will identify the $Q$-equivariant algebra $\FF[P]$ with a suitable exterior algebra.

\begin{example}
For $Q=C_3=\langle q\rangle$ acting on $(\ZZ/2)^2$ nontrivially, the $Q$-action on $\FF_2[(\ZZ/2)^2]\cong \FF_2[\lambda_1,\lambda_2]/(\lambda_1^2,\lambda_2^2)$ is $q(\lambda_1)= \lambda_2$, $q(\lambda_2)= \lambda_1 + \lambda_2 + \lambda_1\lambda_2$.
This $Q$-equivariant algebra is isomorphic to an exterior algebra $\Lambda(V)$ on an equivariant vector space $V=\FF_2y_1\oplus \FF_2y_2$ as follows. Let $Q$ act on $V$ by $q(y_1)=y_2$ and $q(y_2)=y_1+y_2$. Then
\begin{align*}
    \Lambda(V)\to \FF_2[(\ZZ/2)^2]\quad, y_1\mapsto \lambda_1 + \lambda_1\lambda_2,\quad y_2\mapsto \lambda_2 + \lambda_1\lambda_2,
\end{align*}
is an equivariant isomorphism of algebras. Note that it does not preserve the internal gradings of the exterior algebras $\Lambda$ and $\FF_2[(\IZ/2)^2]$. One can verify by hand that there is no equivariant isomorphism $\Lambda(V)\to \FF_2[(\ZZ/2)^2]$ which preserves the grading.
\end{example}

The preceding example can be generalized.

\begin{lemma}\label{lem:algebra_isomorphism}
Let $\FF$ be a field of characteristic two, let $P=(\ZZ/2)^n$ and $Q$ a finite group of odd order acting on $P$. 
Let $I$ denote the maximal ideal in $\FF[P]$. Then there is a
$Q$-equivariant algebra isomorphism $\Lambda(I/I^2)\to \FF[P]$ which respects the augmentations.
\end{lemma}
\begin{proof}
Let $k:I/I^2\to I\subset \FF [P]$ be a $Q$-equivariant section.
Since the target algebra is also an exterior algebra, $k$ extends to an equivariant algebra homomorphism $\Lambda(I/I^2)\to \FF[P]$. It remains to show that it is surjective. This holds since $k$ is a section and any lift of a basis of $I/I^2$ generates $\FF [P]$.

By construction this map sends the augmentation ideal of $\Lambda(I/I^2)$ to the augmentation ideal of $\FF [P]$ and thus it is compatible with the augmentations.
\end{proof}

The classification from \cref{thm:classification_four_dim_perfect} now yields a classification for perfect cochain complexes up to quasi-isomorphism (or equivalently up to isomorphism in the derived category) with small homology for extensions by $(\ZZ/2)^2$.

\begin{corollary}\label{cor:classification_four_dim_perfect} Let $Q$ be a group of odd order acting on $P=\ZZ/2\times \ZZ/2$. There is a bijection between isomorphism classes of perfect cochain complexes over $\FF[P\rtimes Q]$ with four-dimensional homology and triples $(l,L,J)$ where $l$ is an integer, $J\subset \FF[x_1,x_2]$ is a $Q$-invariant parameter ideal and $L$ a one-dimensional $Q$-representation. 
\end{corollary}
\begin{proof}
    The augmented $Q$-algebra $\FF[P]$ is isomorphic to an augmented $Q$-algebra $\Lambda$ on a two-dimensional $Q$-representation $V$ by \cref{lem:algebra_isomorphism}. Thus $\FF[P\rtimes Q]\cong \FF[P]\ast Q\cong \Lambda \ast Q$ and the bijection follows from \cref{thm:classification_four_dim_perfect}.
\end{proof}

Methods of Benson and Carlson construct perfect cochain complexes with trivial action on homology; see \cite[Theorem~4.1]{bensoncarlson1994}.

\begin{remark} For $P=\ZZ/2\times \ZZ/2$, any perfect cochain complex $C^*$ over $\FF[P]$ is either isomorphic to the zero complex in the perfect derived category or has at least four-dimensional homology. Thus if $\dim_\FF H^*(C^*)=4$, then $C^*$ is an indecomposable object in the perfect derived category. Not all indecomposable objects have four-dimensional total homology, e.g., the cochain complex from \cite[Example~3.8]{ruepingstephan2022} has six-dimensional homology and in particular is indecomposable.
\end{remark}

\section{Computing finiteness obstructions}\label{sec:finiteness_obstruction}
As in \cref{sec:spectralsequence_extensions}, consider a short exact sequence
\begin{equation}\label{eq:shortexactsequence}
1\to P\to G\stackrel{pr}{\to} Q\to 1
\end{equation}
of finite groups such that $P$ is a $p$-group and $\FF$ a field of characteristic $p$. In topology, we are interested in cochain complexes of finite, free $G$-CW complexes. These are not just be perfect cochain complexes, but in fact bounded complexes of finitely generated free $\FF[G]$-modules. In this section we consider the finiteness obstruction to determine whether a perfect cochain complex is homotopy equivalent to a finite, free one.

The finiteness obstruction is an element of the reduced Grothendieck group $\tilde{K}_0(\FF [G])$. This group is isomorphic to $\tilde{K}_0(\FF [Q])$ as we will explain below.

\begin{lemma}\label{lem:propOfJ} Let $J\subset \FF[G]$ be the two-sided ideal generated by the augmentation ideal $I\subset \FF [P]$. Then $J$ is the kernel of $pr_*\colon \FF[G] \to \FF[Q]$ and $J$ is nilpotent.
\end{lemma}
\begin{proof} For $p\in P$, the element $p-1$ is in the kernel of $\FF [G]\rightarrow \FF [Q]$ and so is the ideal $I$ generated by these elements. We verify that an arbitrary element $\sum_{g\in G}\lambda_gg$ of this kernel belongs to $J$.
Since    
    \[\sum_{g\in G} \lambda_g g = \sum_{q\in Q} \sum_{g\in pr^{-1}(q)} \lambda_g g\]
    it suffices to show that $\sum_{g\in pr^{-1}(q)} \lambda_g g\in J$ for any fixed $q\in Q$.
    For each $q\in Q$ we have 
    $\sum_{g\in pr^{-1}(q)}\lambda_g=0$. So for a fixed $g'\in pr^{-1}(q)$, we have $\sum_{p\in P} \lambda_{pg'} =0$.
    Thus \[\sum_{p\in P}\lambda_{pg'}pg' = \sum_{p\in P\setminus\{e\}}\lambda_{pg'}(p-1)g'\]
    and this is an element of $J$.
To prove the second statement, let $S$ denote the set $\{p-1\mid p\in P\}$. Since $P$ is normal, we have $\FF [G]\cdot S \cdot \FF [G] = \FF [G]\cdot S$ as sets and inductively $(\FF [G] \cdot S\cdot \FF [G])^n = \FF [G] \cdot(S^n)$. Since the augmentation ideal of a finite $p$-group in modular characteristic is nilpotent, the set $S^n$ is zero for $n$ large enough, and so is $J$.\qedhere
\end{proof}

If $p$ does not divide the order of $Q$ then $I\cdot \FF[G]$ is the Jacobson radical of $\FF[G]$ by \cite[Corollary to Theorem 1]{potter1977}, but we will not use this statement.

We denote the Grothendieck group of isomorphism classes of finitely generated projective modules over a ring $R$ by $K_0(R)$.
\begin{lemma}\label{lem:piindk0iso}
The induced map
\[K_0(pr_*)\colon K_0(\FF [G])\rightarrow K_0(\FF [Q])\]
is an isomorphism.
\end{lemma}
\begin{proof}
Since $\ker(pr_*)$ is nilpotent by \cref{lem:propOfJ}, the induced map on $K_0$ is an isomorphism; see \cite[Lemma~II.2.2]{weibel2013}.
\end{proof}

\begin{definition}
  For a perfect cochain complex $C^*$ over $\FF[G]$, the \emph{Euler characteristic} of $C^*$ is
  \[\chi(C^*)= \sum_i (-1)^i[C^i]\in K_0(\FF [G]).
\] The image $\tilde{\chi}(C^*)$ of $\chi(C^*)$ in the reduced projective class group $\tilde K_0(\FF [G])$ is called the \emph{finiteness obstruction} of $C^*$.
\end{definition}

The element $\tilde{\chi}(C^*)$ vanishes if and only if the perfect cochain complex $C^*$ is homotopy equivalent to a finite, free complex; see \cite{ranicki1985}.

If $p$ does not divide the order of $Q$, then any $\FF [Q]$-module is projective. In particular $\gr(\FF [P])$ from \cref{lem:qactsongrFP} represents a class in $K_0(\FF [Q])$. Moreover, $P\subset G$ is a normal $p$-Sylow subgroup and the short exact sequence \eqref{eq:shortexactsequence} splits by the Schur-Zassenhaus Theorem; see \cite[(8.35)~Theorem]{curtisreiner1981}. In the following result we choose a splitting so that $G$ is a semidirect product $G=P\rtimes Q$ and $[\gr(\FF [P])]=[\res^G_Q \FF [P]]$ in $K_0(\FF [Q])$. If $Q$ is trivial, then the statement reduces to the fact that the ordinary Euler characteristic of a finite, free $\FF [P]$-cochain complex is divisible by the order of $P$. 

\begin{proposition}\label{prop:simplifyingchi}
Let $C^*$ be a perfect $\FF [G]$-cochain complex and assume that $p$ does not divide the order of $Q$. We have
\[pr_*(\chi(C^*))\cdot [\gr(\FF [P])] = \chi (\res_Q^G H^*(C^*))\in K_0(\FF [Q]).\]
Hence if $[\gr(\FF [P])]$ is not a zero-divisor in the ring $K_0(\FF [Q])$, then the finiteness obstruction of $C^*$ vanishes if and only if $\chi(\res^G_Q H^*(C^*))$ lies in the subgroup generated by the product $[\gr(\FF [P])]\cdot [\FF[Q]]$  in $K_0(\FF [Q])$.
\end{proposition}
\begin{proof}
Let $C^*$ be perfect $\FF [G]$-cochain complex and choose a splitting $G=P\rtimes Q$. Since $\FF [Q]$ is semisimple, the Euler characteristic commutes with taking homology and in a short exact sequence the Euler characteristic of the middle term is the sum of the Euler characteristics of the other two terms. It follows that
\[\chi(\res_Q^G H^*(C^*))=\chi(\res_Q^G C^*)=\sum_i \chi (\res_Q^G (F^i C^*/F^{i-1} C^*))\]
for the filtration from \cref{sec:spectralsequence_extensions}. These filtration quotients are the columns $E_0^{i,*}$. By \cref{lem:QactonEL0} and the definition of the functor $K_0$, we obtain
\[\sum_i \chi (\res_Q^G (F^i C^*/F^{i-1} C^*)) = \chi(C^*\otimes_{\FF [G]} \FF [Q]) \cdot [\gr (\FF [P])] = pr_*(\chi(C^*)) \cdot  [\gr (\FF [P])].\]
It follows from \cref{lem:piindk0iso} that the finiteness obstruction $\tilde\chi (C^*)$ vanishes if and only if $pr_*\tilde\chi(C^*)=0$ in  $\tilde{K}_0(\FF [Q])$, i.e., if $pr_*(\chi(C^*))$ lies in the subgroup generated by $\FF [Q]$ of $K_0(\FF [Q])$. If $[\gr(\FF [P])]$ is not a zero-divisor, this holds if and only if $pr_*(\chi(C^*))\cdot [\gr(\FF [P])]=\chi(\res^G_Q H^*(C^*))$ lies in the additive subgroup generated by the product $[\gr(\FF [P])]\cdot [\FF[Q]]$. 
\end{proof}
We will use the following consequence.
\begin{lemma}\label{lem:vanishingofchi}
    Let $C^*$ be a perfect $\FF[G]$-cochain complex and assume that $p$ does not divide the order of $Q$. If $[\gr(\FF[P])]$ is not a zero-divisor in $K_0(\FF[Q])$ and  $\dim_\FF(H^*(C^*))<|G|$, then the following assertions are equivalent:
    \begin{enumerate}
        \item The cochain complex $C^*$ is homotopy equivalent to a finite, free $\FF[G]$-cochain complex;
        \item $\chi(\res^G_Q H^*(C^*))=0$ in $K_0(\FF[Q])$.
    \end{enumerate}
\end{lemma}
\begin{proof}
    By \cref{prop:simplifyingchi} and the assumption on $[\gr(\FF[P])]$,
    the finiteness obstruction vanishes if and only if $\chi(\res^G_{Q} H^*(C^*))$ lies in the additive subgroup generated by $[\gr(\FF[P])]\cdot \FF[Q]$. 
   
    The ring homomorphism $\dim:K_0(\FF [Q])\to K_0(\FF)\cong \ZZ$ maps this cyclic subgroup bijectively to $|G|\ZZ$.
    We have $-|G|< \dim(\chi(\res^G_{C_3} H^*(C^*)))< |G|$ by assumption. Thus $\chi(\res^G_Q H^*(C^*))$ lies in that subgroup if and only if it is zero. 
\end{proof}

If the characteristic $p$ of $\FF$ does not divide the order of $Q$, then $K_0(\FF [Q])$ coincides with the representation ring and is additively isomorphic to $\ZZ^s$, where $s$ is the number of isomorphism classes of irreducible representations and the canonical generators are given by the classes of the irreducible representations.

We specialize to the case of $G=A_4=(\ZZ/2)^2\rtimes C_3$ with $Q=C_3$ acting on $P=\ZZ/2\times \ZZ/2$ nontrivially.

If $X^2+X+1$ does not have a zero in $\FF$, then the only two irreducible representations of $C_3$ are the trivial representation $\FF$ and the two-dimensional representation $V=\FF e_1 \oplus \FF e_2$, where a generator of $C_3$ acts by
$e_1\mapsto e_2$ and $e_2\mapsto e_1+e_2$.

The tensor product $V\otimes V$ decomposes as $V\otimes V \cong V\oplus \FF^2$ so that
\[K_0(\FF [C_3])\cong \ZZ[V]/(V^2-V-2)\]
as rings.

If $X^2+X+1$ has a root $\alpha$, then $X^3+1$ has the three distinct roots $1$, $\alpha$ and $\alpha^2=\alpha+1$ in $\FF$. For each of the roots, we get a one-dimensional representation, where a fixed generator of $C_3$ acts by multiplication with that root. We denote these three one-dimensional representations again by $1,\alpha,\alpha^2$. Since $\alpha\otimes_\FF \alpha \cong \alpha^2$ and $\alpha\otimes_\FF \alpha^2\cong 1$, we obtain 
\[K_0(\FF[C_3])\cong \IZ[\alpha]/(\alpha^3-1).\]

\begin{lemma} \label{lem:grFPnoZeroDivisor} The element $[\gr \FF[P]]$ is not a zero divisor in $K_0(\FF[C_3])$. More precisely:
\begin{enumerate}
    \item If $X^2+X+1$ does not have a zero in $\FF$, then \[[\gr (\FF[P])]= V+2 \text{ in } K_0(\FF[C_3])\cong \ZZ[V]/(V^2-V-2).\]
    \item 
    If $X^2+X+1$ has a zero in $\FF$, then \[[\gr (\FF[P])] = \alpha^2 + \alpha + 2 \text{ in } K_0(\FF[C_3])\cong \ZZ[\alpha]/(\alpha^3-1).\]
\end{enumerate}
\end{lemma} 
\begin{proof}
    Since $C_3$ acts nontrivially on $(\ZZ/2)^2$, there are generators $f_1,f_2$ of $(\ZZ/2)^2$ on which a generator $\tau\in C_3$ acts by $f_1\mapsto f_2$ and $f_2\mapsto f_1f_2$. 
    A basis for $I/I^2$ is given by $[f_1-1], [f_2-1]$. Conjugation by $\tau$ sends this basis to $[f_2-1]$ and \[[f_1f_2-1]=[f_1f_2-1] -[(f_1-1)(f_2-1)]=[f_1-1]+[f_2-1].\] Moreover, conjugation by $\tau$ is the identity on $\FF[P]/I$ and fixes the generator $(f_1-1)(f_2-1)$ of $I^2$.

    The first formula for $[\gr(\FF[P])]$ follows immediately. It is not a zero divisor since $V^2-V-2=(V-2)(V+1)$ in $\ZZ[V]$.
    
    If $X^2+X+1$ has a zero in $\FF$, then the ring homomorphism $K_0(\FF_2[C_3])\to K_0(\FF[C_3])$ sends $V$ to $\alpha + \alpha^2$ from which we deduce the second formula for $[\gr(\FF[P])]$. Since $\alpha^3-1=(\alpha-1)(\alpha^2+\alpha +1)$ in the polynomial ring $\ZZ[\alpha]$, the element $\alpha^2+\alpha+2$ is not a zero divisor in $K_0(\FF[C_3])$.
\end{proof}

\begin{theorem}\label{thm:finitefreecomplexes} Let $C^*$ be a perfect cochain complex over $\FF[A_4]$ with four-dimensional homology corresponding to the triple $(l,L,J)$ as in \cref{cor:classification_four_dim_perfect}.  

Then the following assertions are equivalent:
\begin{enumerate}
\item \label{thm:finitefreecomplexesi}
The cochain complex $C^*$ is homotopy equivalent to a finite, free $\FF[A_4]$-cochain complex;
\item \label{thm:finitefreecomplexesii} $\chi(\res^{A_4}_{C_3} H^*(C^*))=0$ in $K_0(\FF [C_3])$;
\item \label{thm:finitefreecomplexesiii} the graded representation $J/(x_1,x_2)J$ has a $C_3$-invariant basis element of even degree;
\item \label{thm:finitefreecomplexesiv} $J$ has a $C_3$-invariant parameter in even degree.
\end{enumerate}
\end{theorem}
\begin{proof} The first two assertions are equivalent by \cref{lem:vanishingofchi} and \cref{lem:grFPnoZeroDivisor}.

 After shifting, we may assume that the homology of $C^*$ has a basis with elements in degrees $0\leq m_1\leq m_2\leq m_1+m_2$.   Since $\FF[C_3]$ is semisimple, we obtain $\chi(\res^{A_4}_{C_3} H^*(C^*)) = \chi(\res^{A_4}_{C_3} \gr H^*(C^*))$. Since one-dimensional representations are invertible in $K_0$, it follows from \cref{lem:parameteridealFromPerfectComplex}\eqref{lem:parameteridealFromPerfectComplex4} that $\chi(\res^{A_4}_{C_3} \gr H^*(C^*))=0$ if and only if $\chi(\Lambda (\Sigma^{-1}  J/(x_1,x_2)J))=0$.

   We show that $\chi(\Lambda (\Sigma^{-1}  J/(x_1,x_2)J))=0$ if and only if $J/(x_1,x_2)J$ has a trivial one-dimensional subrepresentation of even degree. We distinguish two cases.
   
   First, assume that $J/(x_1,x_2)J$ is the sum of two one-dimensional graded, representations $L_1$ in degree $m_1+1$ and $L_2$ in degree $m_2+1$. We have
   \[\chi(\Lambda (\Sigma^{-1}  J/(x_1,x_2)J))=(1+(-1)^{m_1}[L_1])(1+(-1)^{m_2}[L_2].\]
    Obviously this vanishes if one of the factors is zero, i.e., one $m_i$ is odd and $L_i$ is the trivial representation. We show that it is nonzero otherwise. Neither $2$ nor (if available) the elements $1+\alpha$, $1+\alpha^2$ are zero-divisors in $K_0(\FF[C_3])$, since
    $(1+\alpha)(1+\alpha^2) = 2+\alpha+\alpha^2$ which is not a zero-divisor by \cref{lem:grFPnoZeroDivisor}. 
    Thus we may assume that both $m_i$ are odd and both $L_i$ are nontrivial representations.
    
    Then $X^2+X+1$ has a zero in $\FF$, since otherwise all one-dimensional representations are trivial. Thus $L_i\in \{\alpha,\alpha^2\}$. A case distinction yields:
    \begin{align*}
    (1-\alpha)(1-\alpha) &= 1-2\alpha+\alpha^2\neq 0 \\
    (1-\alpha)(1-\alpha^2)&=2-\alpha-\alpha^2\neq 0\\
    (1-\alpha^2)(1-\alpha^2) &= 1+\alpha-2\alpha^2\neq 0
    \end{align*}
    
    Secondly, assume that $J/(x_1,x_2)J$ is not the sum of two one-dimensional representations, i.e., a two-dimensional irreducible representation in one single degree $m+1$. In that case \eqref{thm:finitefreecomplexesiii} does not hold. Note that $X^2+X+1$ does not have a zero in $\FF$, and we have 
    \[\chi(\Lambda (\Sigma^{-1}  J/(x_1,x_2)J))=2+(-1)^mV\neq 0 \in K_0(\FF [C_3]). \]

    Finally, we show that \eqref{thm:finitefreecomplexesiii} is equivalent to \eqref{thm:finitefreecomplexesiv}. A $C_3$-invariant parameter in $J$ generates a one-dimensional trivial subrepresentation of $J/(x_1,x_2)J$. Conversely, applying the Reynolds operator $\mathcal R(z)=1/3\sum_{q \in C_3} qz$ to a representative of a generator of a $C_3$-invariant subspace of $J/(x_1,x_2)J$ yields a $C_3$-invariant parameter. 
\end{proof}
For $\FF=\FF_2$ we can read off from the homology of $C^*$ whether it is homotopy equivalent to a finite, free one:
\begin{corollary}\label{cor:finitefreecomplexesforF2}
 A perfect $\FF_2[A_4]$-cochain complex $C^*$ such that its homology is four-dimensional with basis elements in degrees $0\leq m\leq n\leq t$ is homotopy equivalent to a finite, free cochain complex if and only if $m$ or $n$ is odd and $H^*(C^*)$ is a trivial $C_3$-representation.
\end{corollary}
\begin{proof}
Over $\FF_2$ all one-dimensional $\FF_2[C_3]$-representations are trivial. Moreover, recall that $t=m+n$ by \cref{lem:parameteridealFromPerfectComplex}\eqref{lem:parameteridealFromPerfectComplexv}. It follows immediately from \cref{thm:finitefreecomplexes}\eqref{thm:finitefreecomplexesii} that $C^*$ is homotopy equivalent to a finite, free cochain complex if $m$ or $n$ is odd and $C_3$ acts trivially on $H^*(C^*)$.

Conversely, if $C^*$ is homotopy equivalent to a finite, free cochain complex, then $m$ or $n$ is odd by \cref{thm:finitefreecomplexes}\eqref{thm:finitefreecomplexesiv} and \cref{lem:parameteridealFromPerfectComplex}\eqref{lem:parameteridealFromPerfectComplex4}. 
\end{proof}

\cref{cor:finitefreecomplexesforF2} is an algebraic result with a topological application. The following argument is our original proof of \cite[Theorem~7.1]{ruepingstephanyalcin2022}.
\begin{remark}  In \cite[Theorem~2]{oliver1979}, Oliver proved that $A_4$ can not act freely on a finite CW complex $X$ with cohomology ring $H^*(X;\ZZ)\cong H^*(S^n\times S^n;\ZZ)$. The statement of \cite[Theorem~7.1]{ruepingstephanyalcin2022} is that the assumption on the cohomology ring can be weakened to $\FF_2$-coefficients. Indeed, by \cite[Theorem~1]{oliver1979}, it suffices to show that $A_4$ can not act freely on a finite $CW$-complex $X$ with cohomology ring $H^*(X;\FF_2)\cong H^*(S^n\times S^n;\FF_2)$ on which $A_4$ acts nontrivially. This holds by \cref{thm:finitefreecomplexes} applied to the cellular cochain complex of $X$.
\end{remark}

\section{A topological application to free \texorpdfstring{$A_4$}{A4}-actions}
We will show in \cref{thm:rigidity} that cochain complexes of finite, free $A_4$-CW complexes with four-dimensional cohomology are rigid. They are determined by the degrees of the nonzero cohomology groups. 

We begin with a comparison of the topological and algebraic Borel construction.
For any short exact sequence of finite groups
\[1\to N\to G\to Q\to 1
\]
and field $k$, we have equipped the group cohomology $H^*(BN;k)$ with a $Q$-action. More generally, $\Ext^*_{k[N]}(k,D^*)$ has a $Q$-action for any cochain complex $D^*$ of right $k[G]$-modules by \cref{lem:equivariantactionOfExt}. For the cochains $C^*=C^*(X;k)$ on a left $G$-CW complex $X$, this can be modeled topologically. The quotient of the universal left $G$-space $EG$ by $N$ is a model for $BN$ on which $Q$ acts. More generally, for the left $G$-space $X$, the Borel construction $(X)_{hN}=EG\times_N X$ inherits a left $Q$-action defined by $q[e,x]\coloneqq [ge,gx]$ for $q=Ng=gN$.

Let $\varepsilon^*(G)$ be a projective resolution of $k$ over $k[G]$. Then $\Hom_{k[N]}(\varepsilon^*(G), D^*)$ is the algebraic Borel construction from \cite[Section~1.2]{alldaypuppe1993} extended to extensions. To connect to the topological Borel construction, note that for $C^*=\Hom_k(C_*,k)$, i.e., the cochain complex of the singular chain complex $C_*=C_*(X;k)$, we have:
\[\Hom_{k[N]}(\varepsilon^*(G),\Hom_\FF(C_*,k))\cong \Hom_k (\varepsilon^*(G)\otimes_{k[N]} C_*,k)
\]

In \cref{lem:equivariantactionOfExt}, we have equipped $\Ext_{k[N]}^*(k, D^*)$ with a $Q$-equivariant action by the group cohomology $\Ext_{k[N]}^*(k,k)=H^*(BN;k)$ using Yoneda composition. Alternatively, the action can be defined with cup products using a diagonal approximation as in \cite{alldaypuppe1993}; see \cite[Lemma~3.2.3]{benson1991I}. We used Yoneda composition since this description makes clear that the action only depends on the structure of $k[N]$ as an augmented $k$-algebra and thus does not depend on the comultiplication of the Hopf algebra structure.

If the $N$-action on $X$ is free, then $H^*(X/N;k)\cong H^*(EG\times_N X; k)$ and the $H^*(BN;k)$-action agrees with the action induced by $H^*(-;k)$ applied to the classifying map $f\colon X/N\to BN$. 

\begin{lemma}\label{lem:idealissteenrodclosed}
Let $X$ be a finite $A_4$-CW complex with four-dimensional cohomology $H^*(X;\FF_2)$ such that the restriction of the $A_4$-action to $P=\ZZ/2\times \ZZ/2$ is free. Then the kernel of $H^*(BP;\FF_2)\to H^*(X/P;\FF_2)$ is a Steenrod closed, $C_3$-invariant parameter ideal.
\end{lemma}
\begin{proof}
    Let $C^*=C^*(X;\FF_2)$ be the cellular cochain complex of $X$ and $J$ the annihilator ideal of $\Ext^*_{\FF_2[P]}(\FF_2, C^*)$. Then $J$ is $C_3$-invariant by \cref{lem:equivariantactionOfExt} and a parameter ideal by \cref{lem:parameteridealFromPerfectComplex}.  Since the $H^*(BP)$-module structure of $H^*(X/P)\cong \Ext_{\FF[P]}^*(\FF, C^*)$ is induced by a ring homomorphism $H^*(BP)\to H^*(X/P)$, the annihilator ideal $J$ is the kernel of this homomorphism. The ring homomorphism is induced by a map of spaces and thus commutes with Steenrod operations. Hence $J$ is closed under Steenrod operations.
\end{proof}

\begin{theorem}\label{thm:rigidity} If there exists a finite, free $A_4$-CW complex $X$ with four-dimensional total cohomology $H^*(X;\FF_2)$ with a basis in degrees $0\leq m\leq n\leq t$, then its cellular cochain complex is determined by $m$ and $n$ up to homotopy. 
\end{theorem}
\begin{proof}
    By \cref{cor:classification_four_dim_perfect} for $\FF=\FF_2$, the cellular cochain complex $C^*\coloneqq C^*(X;\FF_2)$ is determined by just its corresponding parameter ideal $J\subset H^*(BP)$ for $P=\ZZ/2\times \ZZ/2$
    since for spaces the lowest degree $l$ with nonzero homology is $l=0$ and every one-dimensional $C_3$-representation over $\FF_2$ is trivial.
    
    Since $C^*(X;\FF_2)$ is a finite, free $\FF_2[A_4]$-cochain complex, the parameter ideal $J$ has a $C_3$-invariant parameter $x$ by \cref{thm:finitefreecomplexes}. Over $\FF_2$, the graded two-dimensional representation $J/(x_1,x_2)J$ has to be trivial. Thus we can also find a second $C_3$-invariant parameter $y$ using the Reynolds operator (as in the proof of \cref{thm:finitefreecomplexes}).
    Let $J'\subset H^*(BP)^{C_3}$ be the ideal generated by $x,y\in  H^*(BP)^{C_3}=H^*(BA_4)$. In particular, $J$ is the extension of $J'$ to $H^*(BP)$.
    Then $J'$ is a parameter ideal; see \cite[Lemma~2.3]{ruepingstephanyalcin2022}. The ideal $J$ is Steenrod closed by \cref{lem:idealissteenrodclosed} and so is $J'$ by \cite[Lemma~2.8]{ruepingstephanyalcin2022}.
    
Finally, \cite[Corollary~6.13]{ruepingstephanyalcin2022} states that there is at most one Steenrod closed parameter ideal in $H^*(BA_4)$ with parameters in degrees $m+1,n+1$.
\end{proof}
\cref{thm:rigidity} does not hold for $P=\ZZ/2\times \ZZ/2$ instead of $A_4$.
\begin{example} Let $P=\ZZ/2\times \ZZ/2$. Consider the product action of the antipodal actions on $X=S^2\times S^3$. Then the cohomology ring of the quotient is\[H^*(\RR P^2\times \RR P^3; \FF_2)\cong \FF_2[x_1, x_2]/( x_1^3, x_2^4).\] Thus the corresponding parameter ideal of the cellular cochain complex of $X$ is $( x_1^3, x_2^4)$. On the other hand, Oliver \cite{oliver1979} constructed a free $A_4$-action on $S^2\times S^3$. The only Steenrod closed parameter ideal in $H^*(BA_4;\FF_2)$ with parameters of degrees $3$ and $4$ is $(x_1x_2(x_1+x_2),x_1^4+(x_1x_2)^2+x_2^4)$. Its extension to $\FF_2[x_1,x_2]$ classifies the cochain complex of the restriction from $A_4$ to $\ZZ/2\times \ZZ/2$. Since the two ideals are different, the cellular cochain complexes of the two $P$-actions are not homotopy equivalent. Note that the product action can not be extended to an $A_4$-action since the corresponding parameter ideal is not $C_3$-invariant.
\end{example}
We do not know if rigidity holds topologically.
\begin{question}
Given two finite, free $A_4$-CW complexes homotopy equivalent to $S^m\times S^n$. Are they $A_4$-equivariantly homotopy equivalent?
\end{question}


\providecommand{\bysame}{\leavevmode\hbox to3em{\hrulefill}\thinspace}
\providecommand{\MR}{\relax\ifhmode\unskip\space\fi MR }
\providecommand{\MRhref}[2]{%
  \href{http://www.ams.org/mathscinet-getitem?mr=#1}{#2}
}
\providecommand{\href}[2]{#2}

\end{document}